\documentclass{amsart}

\usepackage{amsthm,amsfonts,amsmath,amssymb}
\usepackage[all]{xy}
\usepackage{lscape}
\usepackage{lunadiagrams}
\usepackage{hyperref}
\usepackage{pdfpages}

\theoremstyle{plain}

\newtheorem{theorem}{Theorem}[section]
\newtheorem{lemma}[theorem]{Lemma}     
\newtheorem{corollary}[theorem]{Corollary}
\newtheorem{proposition}[theorem]{Proposition}

\theoremstyle{definition}

\theoremstyle{remark}

\newtheorem{remark}[theorem]{Remark}

\numberwithin{equation}{section}

\DeclareMathOperator{\Pic}{Pic}

\DeclareMathOperator{\Hom}{Hom}

\DeclareMathOperator{\height}{ht}
\DeclareMathOperator{\supp}{supp}

\DeclareMathOperator{\Spec}{Spec}
\DeclareMathOperator{\diag}{diag}

\newcommand{\calA}{{\mathcal A}}

 \newcommand{\calL}{\mathcal L}
\newcommand{\calM}{\mathcal M} 
\newcommand{\calO}{\mathcal O} 
\newcommand{\calX}{\mathcal X}

\newcommand{\mC}{\mathbb C} \newcommand{\mN}{\mathbb N}
\newcommand{\mP}{\mathbb P} 
 \newcommand{\mZ}{\mathbb Z}

\newcommand{\gog}{\mathfrak g}

\newcommand{\gok}{\mathfrak k}

\newcommand{\gop}{\mathfrak p}

\newcommand{\got}{\mathfrak t}
\newcommand{\goz}{\mathfrak z}

\newcommand{\sfA}{\mathsf A} \newcommand{\sfB}{\mathsf B}
\newcommand{\sfC}{\mathsf C} \newcommand{\sfD}{\mathsf D}

\newcommand{\gra}{\alpha}     \newcommand{\grg}{\gamma}
 \newcommand{\grl}{\lambda}  \newcommand{\grs}{\sigma}
 \newcommand{\gro}{\omega}

\newcommand{\grG}{\Gamma} \newcommand{\grD}{\Delta}  \newcommand{\grL}{\Lambda}
 \newcommand{\grS}{\Sigma}

\newcommand{\mss}{\mathrm{ss}}

\newcommand{\ra}         {\rightarrow}
\newcommand{\lra}        {\longrightarrow}

\newcommand{\vuoto}      {\varnothing}

\renewcommand{\geq}      {\geqslant}
\renewcommand{\leq}      {\leqslant}
\newcommand{\senza}      {\smallsetminus}

\newcommand{\ol}         {\overline}
\newcommand{\ul}         {\underline}
\newcommand{\wt}         {\widetilde}

\newsavebox{\kdwzero}
\newsavebox{\kdwone}
\newsavebox{\kdwtwo}
\savebox{\kdwzero}{\put(-240,240){0}}
\savebox{\kdwone}{\put(-240,240){1}}
\savebox{\kdwtwo}{\put(-240,240){2}}

\title[Spherical nilpotent orbits in complex symmetric pairs]{Regular functions on spherical nilpotent\\orbits in complex symmetric pairs:\\classical Hermitian cases}

\email{bravi@mat.uniroma1.it}

\curraddr{\textsc{Dipartimento di Matematica\\ Sapienza Universit\`a di Roma\\ 
Piazzale Aldo Moro 5\\ 00185 Roma, Italy}}

\email{jacopo.gandini@sns.it}

\curraddr{\textsc{Scuola Normale Superiore\\
Piazza dei Cavalieri 7\\ 56126 Pisa, Italy}}

\author{Paolo Bravi, Jacopo Gandini}



\begin{document}

\begin{abstract}
Given a classical semisimple complex algebraic group $G$ and a symmetric pair $(G,K)$ of Hermitian type, we study the closures of the spherical nilpotent $K$-orbits in the isotropy representation of $K$. We show that all such orbit closures are normal and describe the $K$-module structure of their ring of regular functions.
\end{abstract}

\maketitle

\section*{Introduction}
Let $G$ be a connected semisimple complex algebraic group, and let $K$ be the fixed point subgroup of an algebraic involution $\theta$ of $G$. 

The Lie algebra $\mathfrak g$ of $G$ splits into the sum of eigenspaces of $\theta$, 
\[\mathfrak g=\mathfrak k \oplus \mathfrak p,\]
where the Lie algebra $\mathfrak k$ of $K$ is the eigenspace of eigenvalue 1, and $\mathfrak p$ is the eigenspace of eigenvalue $-1$. The latter is called the isotropy representation of $K$.

In the present paper, we continue the systematic study initiated in \cite{BCG} of the spherical nilpotent $K$-orbits in $\mathfrak p$, classified by King \cite{Ki04}.

Here we treat the classical symmetric pairs $(G,K)$ of Hermitian type: in particular, $K$ is a maximal Levi subgroup of $G$, and $\mathfrak p = \mathfrak p_1\oplus\mathfrak p_2$ splits into the sum of two simple $K$-modules dual to each other. Under this assumption, we prove that all spherical nilpotent orbit closures are normal (Theorem~\ref{teo:allnormal}), and we compute the $K$-module structure of their coordinate rings.

In Appendix \ref{A} we report the list of the spherical nilpotent $K$-orbits in $\mathfrak p$ for all classical symmetric pairs $(\mathfrak g,\mathfrak k)$ of Hermitian type. In the list, every orbit is labelled with the corresponding signed partition, \cite{CoMG}.

For every orbit we provide an explicit description of a representative $e\in\mathfrak p$, as an element of a normal triple $\{h,e,f\}$, and the centralizer of $e$, which we denote by $K_e$. All these data can be deduced from \cite{Ki04}.

Then we provide the Luna spherical system associated with $\mathrm N_K(K_e)$, the normalizer of $K_e$ in $K$, which is a wonderful subgroup of $K$. 

Let us write $e=e_1+e_2$ with $e_1\in\mathfrak p_1$ and $e_2\in\mathfrak p_2$. If $e_1$ and $e_2$ are both non-zero, then the orbit $Ke$ is a \emph{bicone}, the normalizer of $K_e$ is the common stabilizer of the lines $[e_1] \in \mP(\gop_1)$ and $[e_2] \in  \mP(\gop_2)$, and $\mathrm N_K(K_e)/K_e$ is a 2-dimensional complex torus.

The Luna spherical systems are used to deduce the normality of the $K$-orbit closures, and to compute the $K$-module structure of the corresponding coordinate rings.

In the tables of Appendix \ref{B} we summarize the results of our computations. In Tables~1--6 we describe the $K$-module structure of $\mathbb C[\ol{Ke}]$ by giving a set of generators of its weight semigroup $\grG(\ol{Ke})$ (that is, the semigroup of the highest weights occurring in $\mathbb C[\ol{Ke}]$). Tables~7--11 contain the Luna spherical system of $\mathrm N_K(K_e)$.

In Section~\ref{s:3} we adapt the criterion of normality for spherical cones used in our previous papers \cite{BGM} and \cite{BCG} to the case of a spherical multicone. In Section~\ref{s:1} we compute the Luna spherical systems of the normalizers $\mathrm N_K(K_e)$. In Section~\ref{s:2} we show that the multiplication of sections of globally generated line bundles on the corresponding wonderful varieties is surjective. In Section~\ref{s:4} we deduce our results on normality and semigroups. 

\subsection*{Acknowledgments}

We would like to thank Andrea Maffei for his help in this project and, in particular, for sharing with us his ideas which led to Section~\ref{s:3}. We would like to thank also the anonymous referee for some useful remarks.
  
\subsection*{Notation}

As in our previous paper, simple roots of irreducible root systems are denoted by $\alpha_1,\alpha_2,\ldots$ and enumerated as in Bourbaki, when belonging to different irreducible components they are denoted by $\alpha_1,\alpha_2,\ldots$, $\alpha'_1,\alpha'_2,\ldots$, $\alpha''_1,\alpha''_2,\ldots$, and so on. When $G$ (resp. $K$, $T$, ...) is an algebraic group, we will denote the associated Lie algebra by the corresponding fraktur character $\gog$ (resp. $\gok$, $\got$, ...). If moreover $X$ is a $G$-variety (resp. a $K$-variety) and $x \in X$, we will denote the stabilizer of $x$ by $G_x$ (resp. $K_x$).

\section{A criterion for the normality of a spherical multicone}\label{s:3}

In this section, $G$ will denote a connected reductive complex algebraic group (possibly not semisimple). By generalizing an argument due to C.~De~Concini in \cite{CDC}, in this section we will give a criterion (Theorem~\ref{teo:normalita}) to test the normality of a spherical multicone. We will apply it in Section \ref{s:4} to study the normality of closures of spherical $K$-orbits in $\gop$, where $K$ is a symmetric subgroup of a semisimple group and $\gop$ is the isotropy representation of $K$.

Fix a Borel subgroup $B \subset G$ and a maximal torus $T \subset B$. Let $\Pi = \{\grl_1, \ldots, \grl_m\} \subset \calX(T)$ be a finite set of dominant weights and denote $V_\Pi = \bigoplus_{1=1}^m V(\grl_i)$. For all $i= 1, \ldots, m$ we denote by $\hat\pi_i \colon V \lra V(\grl_i)$ the corresponding projection.

Let $e \in V_\Pi$ and suppose that $Ge$ is spherical. We will assume that $\hat\pi_i(e) \neq 0$ for all $i=1, \ldots, m$. Under this assumption, we have a well-defined equivariant map $\pi_i \colon Ge \lra \mP(V(\grl_i))$ for all $i$, hence diagonally we get an equivariant map
$$\pi\colon Ge \lra \mP(V(\grl_1)) \times \ldots \times \mP(V(\grl_m)).$$

We say that a closed subvariety $Z \subset V_\Pi$ is a \textit{multicone} (w.r.t.\ the given decomposition of $V_\Pi$) if, for all $z \in Z$ and $(\xi_1,\ldots,\xi_m) \in \mC^m$, it holds 
\[\xi_1\hat\pi_1(z)+\ldots+\xi_m\hat\pi_m(z) \in Z.\] 
Given a spherical orbit $Ge \subset V_\Pi$, we will define a wonderful $G$-variety $X$ endowed with a map
$$
	\phi \colon X \lra \mP(V(\grl_1)) \times \ldots \times \mP(V(\grl_m))
$$
which is birational on its image and which identifies $\Pi$ with a set of globally generated line bundles on $X$. If moreover $\ol{Ge} \subset V_\Pi$ is a multicone, we will establish a combinatorial criterion for $\ol{Ge}$ to be normal in terms of $\Pi$, regarded as a subset of $\Pic(X)$.

\begin{proposition}  \label{prop: stabilizzatori}
Let $Ge \subset V_\Pi$ be a spherical orbit, then $G_{\pi(e)} = \mathrm N_G(G_e)$.
\end{proposition}

\begin{proof}
Write $e = \sum_{i=1}^m e_i$ with $e_i \in V(\grl_i) \subset V_\Pi$, and notice that $G_e = G_{e_1} \cap \ldots \cap G_{e_m}$. Being spherical, for all $i=1,\ldots,m$ the subgroup $G_e$ fixes pointwise a unique line in $V(\grl_i)$, namely $[e_i]$. If $g \in \mathrm N_G(G_e)$, it follows that $ge_i \in V(\grl_i)^{G_e}$ for all $i=1,\ldots,m$, hence $g \in G_{[e_1]} \cap \ldots \cap G_{[e_m]} = G_{\pi(e)}$. Therefore $\mathrm N_G(G_e) \subset G_{\pi(e)}$. Similarly, if $g \in G_{\pi(e)}$, then $ge_i \in V(\grl_i)^{G_e}$ for all $i=1,\ldots,m$, hence $ge \in V_\Pi^{G_e}$ and $g^{-1}hg e = e$ for all $h \in G_e$, that is $g\in \mathrm N_G(G_e)$. This shows the equality $G_{\pi(e)} = \mathrm N_G(G_e)$.
\end{proof}

Given a spherical subgroup $H \subset G$, it follows by a theorem of F.~Knop that the homogeneous space $G/\mathrm N_G(H)$ admits a wonderful compactification (see \cite[Corollary 6.5]{Kn} or \cite[Proposition 2.4]{BP15}). Notice that in general the normalizer of a spherical subgroup is not self-normalizing (see e.g.\ \cite[Example~4]{Avd}). The spherical homogeneous spaces occurring in our cases~2.4 and~4.4 give other examples of this kind.

Let us now fix some notation and recall some general facts on line bundles and their sections on a wonderful variety $X$, see e.g. \cite[Section 1]{BGM} and the references therein. Let $\grD$ be the set of colors of $X$ (that is, the $B$-stable prime divisors of $X$ which are not $G$-stable) and if $D \in \mZ \grD$ denote by $\calL_D \in \Pic(X)$ the corresponding line bundle. Recall that $\Pic(X)$ is a free Abelian group with basis the line bundles $\calL_D$ with $D \in \grD$, and that, under the indentification $\Pic(X) = \mZ\grD$, the semigroup of globally generated line bundles is identified with $\mN \grD$. Given $\calL, \calL' \in \Pic(X)$ we denote by
$$
	m_{\calL, \calL'} : \grG(X, \calL) \otimes \grG(X, \calL') \lra \grG(X, \calL \otimes \calL') 
$$
the multiplication of sections. If $\calL = \calL_D$ and $\calL' = \calL_E$ with $D,E \in \mZ \grD$, we will also denote $m_{\calL, \calL'}$ by $m_{D, E}$.

Let $Y \subset X$ be the unique closed $G$-orbit and let $y_0 \in Y$ be the unique $B^-$-fixed point, where $B^-$ denotes the opposite Borel subgroup of $B$. Recall the set of the spherical roots of $X$, defined by 
$$
	\grS = \{T \text{-weights in } \mathrm T_{y_0} X/\mathrm T_{y_0} Y \}.
$$
Then there is a bijective correspondence between spherical roots and $G$-stable divisors in $X$, which allows to regard $\mZ \grS$ as a sublattice of $\Pic(X) = \mZ\grD$.
This defines a partial order $\leq_\grS$ on $\mN \grD$ as follows: if $D,E \in \mN \grD$, then $D \leq_\grS E$ if and only if $E-D \in \mN\grS$. We say that $E \in \mN\grD$ is \textit{minuscule} w.r.t.\ $\leq_\grS$ if there exists no $F \in \mN \grD$ such that $F \leq_\grS E$ and $F \neq E$. Notice that $\grG(X,\calL_E)$ is
an irreducible $G$-module if and only if $E$ is minuscule in $\mN\grD$ w.r.t.\ $\leq_\grS$.
Given $F \in \mN\grD$, let indeed $V_F \subset \grG(X, \calL_F)$ be the $G$-module generated by a canonical section, and given $\grg \in \mZ \grS$ (regarded as a sublattice of $\mZ \grD$) let $s^\grg \in \grG(X,\calL_\grg)^G$. Then for $E \in \mN \grD$ we have the following decomposition into simple $G$-modules
\begin{equation}	\label{eq:decomposition}
	\grG(X,\calL_E) = \bigoplus_{F \leq_\grS E} s^{E-F} V_F.
\end{equation}

Going back to our setting, suppose that $Ge \subset V_\Pi$ is a spherical orbit. Then $G/G_{\pi(e)}$ admits a wonderful compactification $X$. Notice that the center $Z(G)$ of $G$ acts trivially on $G/G_{\pi(e)}$, hence it acts trivially on $X$ as well. 

For all $i=1, \ldots, m$, let $\phi_i \colon G/G_{\pi(e)} \lra \mP(V(\grl_i))$ be the restriction of the projection defined on $\mP(V(\grl_1)) \times \ldots \times \mP(V(\grl_m))$.  By the theory of spherical embeddings (see in particular \cite[Theorem 5.1]{Kn0}), any dominant morphism $G/G_{\pi(e)} \lra Y$ to a complete spherical $G$-variety $Y$ with a unique closed orbit can be extended to an equivariant morphism $X\lra Y$. This in particular applies to the closure of $\phi_i(G/G_{\pi(e)})$, since $\mathbb P ( V ( \lambda _i ))$ contains a unique $B$-fixed point that is contained in every closed $G$-orbit. Thus $\phi_i$ extends to an equivariant morphism $\phi_i \colon X \lra \mP(V(\grl_i))$, that we still denote by the same symbol by abuse of notation. Mapping $X$ diagonally to $\mP(V(\grl_1)) \times \ldots \times \mP(V(\grl_m))$ via $\phi = (\phi_1, \ldots, \phi_m)$, we get then an equivariant morphism
$$
	\phi \colon X \lra \mP(V(\grl_1)) \times \ldots \times \mP(V(\grl_m)).
$$
which extends the inclusion $G/G_{\pi(e)} \lra \pi(Ge) \subset \mP(V(\grl_1)) \times \ldots \times \mP(V(\grl_m))$.

For all $i=1, \ldots, m$, let $\calL_i \in \Pic(X)$ be the globally generated line bundle defined by setting $\calL_i = \phi_i^* \calO(1)$, and let $D_i \in \mN \grD$ be the unique $B$-stable divisor such that $\calL_{D_i} = \calL_i$. We denote
$$\grD_\Pi(e) = \{D_1, \ldots, D_m\}.$$

For every $(d_1,\ldots,d_m) \in \mN^m$, according with equation \eqref{eq:decomposition} the $G$-module $\grG(X,\calL_1^{d_1} \otimes \ldots \otimes \calL_m^{d_m})$ decomposes into the direct sum of the simple modules with highest weights of the form 
\begin{equation}\label{eq:weights}
d_1\lambda_1^*+\ldots+d_m\lambda_m^*-(d_1D_1+\ldots+d_mD_m-F)
\end{equation} 
where $F\in\mathbb N\Delta$ and $F\leq_{\Sigma}d_1D_1+\ldots+d_mD_m$.
Notice that we denote by $\lambda^*$ the highest weight of $V(\lambda)^*$, the dual of the simple module of highest weight $\lambda$. 

We will prove the following.

\begin{theorem}	\label{teo:normalita}
Suppose that the orbit $Ge$ is spherical and its closure $\ol{Ge} \subset V_\Pi$ is a multicone, and assume that the multiplication of sections $m_{\calL, \calL'}$ is surjective for all globally generated $\calL, \calL' \in \Pic(X)$. Then $\ol{Ge} \subset V$ is normal if and only if every $D \in \grD_\Pi(e)$ is minuscule in $\mN\grD$.
\end{theorem}

Denote by $A(Ge)$ the coordinate ring of $\overline{Ge}$ and by $A_d(Ge)$ its component of degree $d$, then $A(Ge)$ is generated by $A_1(Ge) = \bigoplus_{\grl \in \Pi} V(\grl)^*$.
Consider the multigraded ring
$$
	\widetilde A(Ge) = \bigoplus_{(d_1,\ldots,d_m) \in \mN^m} \grG(X,\calL_1^{d_1} \otimes \ldots \otimes \calL_m^{d_m}).
$$
and denote $\wt{Ge} = \Spec \widetilde A(Ge)$. Setting
$$\widetilde A_d(Ge) = \bigoplus_{d_1+\ldots+d_m = d} \grG(X,\calL_1^{d_1} \otimes \ldots \otimes \calL_m^{d_m})
$$
we get a canonical inclusion $A_d(Ge) \subset \widetilde A_d(Ge)$: indeed for all $i=1,\ldots,m$ the $G$-module $V(\grl_i)^*$ is canonically identified with a submodule in $\grG(X,\calL_i) \subset \widetilde A_1(Ge)$. This makes canonically $A(Ge)$ a subring of $\widetilde A(Ge)$, and we get a projection
$$
	p\colon \wt{Ge} \lra \ol{Ge}
$$

Theorem~\ref{teo:normalita} is a consequence of the following description of the normalization of $\ol{Ge}$.

\begin{theorem} \label{teo:normalizzazione} 
Let $Ge$ be spherical and $\ol{Ge} \subset V_\Pi$ be a multicone, and suppose that the multiplication of sections $m_{\calL, \calL'}$ is surjective for all $\calL, \calL' \in \Pic(X)$. Then $p \colon \widetilde{Ge} \lra \ol{Ge}$ is the normalization map.
\end{theorem}

\begin{remark}
In \cite{BCGM} a standard monomial theory for the Cox ring of a wonderful variety was constructed. By making use of such a tool, reasoning as in \cite[Proposition 5.2]{BCGM} one can actually show that $\wt{Ge}$ has rational singularities, hence it is in particular normal. We will give anyway a direct proof of the normality of $\widetilde{Ge}$ by using an easy geometric argument (see also \cite[Section~7]{BGM}).
\end{remark}

\begin{proof}[Proof of Theorem~\ref{teo:normalita}]
By Theorem~\ref{teo:normalizzazione} it follows that $\ol{Ge}$ is normal if and only if $A(Ge) = \widetilde A(Ge)$. Since the multiplication of sections $m_{\calL, \calL'}$ is surjective for all globally generated $\calL, \calL' \in \Pic(X)$, it follows that $\widetilde A(Ge)$ is generated by its degree one component $\widetilde A_1(Ge)$. Therefore $A(Ge) = \widetilde A(Ge)$ if and only if $A_1(Ge) = \widetilde A_1(Ge)$, if and only if $\grG(X,\calL_i) = V(\grl_i)^*$ for all $i=1, \ldots,m$. Being $\calL_i = \calL_{D_i}$, by the description of the irreducible components of $\grG(X,\calL_{D_i})$ (see e.g.\ \cite[Proposition~1.1]{BGM}), this is equivalent to the fact that every $D_i$ is minuscule in $\mN\grD$. 
\end{proof}

\subsection{Proof of Theorem~\ref{teo:normalizzazione}}

We will split the proof of Theorem~\ref{teo:normalizzazione} in several propositions. The following shows that $\widetilde A(Ge)$ is a normal ring.

Suppose that $Y$ is a normal variety and let $\calM_1, \ldots, \calM_k \in \Pic(Y)$, we denote
$$
	A(\calM_1, \ldots, \calM_k) = \bigoplus_{(d_1,\ldots,d_k) \in \mN^k} \grG(Y,\calM_1^{d_1} \otimes \ldots \otimes \calM_k^{d_k})
$$
and define correspondingly a sheaf of $\calO_Y$-algebras on $Y$ by setting $$\calA(\calM_1, \ldots, \calM_k) = \bigoplus_{(d_1,\ldots,d_k) \in \mN^k} \calM_1^{d_1} \otimes \ldots \otimes \calM_k^{d_k}.$$

\begin{proposition}
Suppose that $Y$ is a normal variety and let $\calM_1, \ldots, \calM_k \in \Pic(Y)$. Then $A(\calM_1, \ldots, \calM_k)$ is a normal ring.
\end{proposition}

\begin{proof}
Notice that $A(\calM_1, \ldots, \calM_k)$ is a domain since $Y$ is irreducible. Suppose first that $Y$ is affine and that $\calM_i \simeq \calO_Y$ for all $i=1, \ldots, k$. Then $A(\calM_1, \ldots, \calM_k) = \mC[Y][s_1, \ldots, s_k]$ is a polynomial ring with coefficients in $\mC[Y]$, hence it is normal since $Y$ is so.

Consider now the general case. Let $s_1, s_2 \in A(\calM_1, \ldots, \calM_k)$ and suppose that $s_1/s_2$ is integral over $A(\calM_1, \ldots, \calM_k)$. Let $U\subset Y$ be an affine open subset such that $\calM_i|_U$ is trivial for all $i$. Then $(s_1/s_2)|_U$ is integral over $A(\calM_1|_U, \ldots, \calM_k|_U)$, thus it belongs to $A(\calM_1|_U, \ldots, \calM_k|_U)$. The claim follows since we can cover $Y$ with affine open subsets $U \subset Y$ such that $\calM_i|_U$ is trivial for all $i=1,\ldots,k$.
\end{proof}

In particular we have
$$\wt A(Ge) = A(\calL_1, \ldots, \calL_m) = \grG(X,\calA(\calL_1, \ldots, \calL_m)).$$

Notice that, for all $i=1,\ldots,m$, the map $\phi_i \colon X \lra \mP(V(\grl_i))$ factors through $\mP(\grG(X,\calL_i)^*)$. Therefore the map $\phi = (\phi_1, \ldots, \phi_m)$ factors as follows
$$
	X \stackrel{\wt\phi}{\lra} \mP(\grG(X,\calL_1)^*) \times \ldots \times \mP(\grG(X,\calL_m)^*) \stackrel{\psi}{\dashrightarrow} \mP(V(\grl_1)) \times \ldots \times \mP(V(\grl_m)),
$$
where $\psi = (\psi_1, \ldots, \psi_m)$ is the canonical projection.

Consider the multiplication map
$$
	\bigoplus_{n\geq 0} \mathsf S^n \big( \wt A_1(Ge) \big) \lra A(\calL_1, \ldots, \calL_m):
$$
it is surjective by the assumption of Theorem~\ref{teo:normalizzazione}, and its kernel coincides with the homogeneous ideal of $\wt\phi(X)$. It follows that $\wt{Ge}$ is the multicone over $\wt\phi(X)$, whereas $\ol{Ge}$ is by assumption the multicone over $\phi(X)$.

By Proposition~\ref{prop: stabilizzatori} the map $\phi \colon X \lra \phi(X)$ is birational. Therefore the restriction of $\pi$ induces a birational map $\wt\phi(X) \lra \phi(X)$, and taking the affine multicones it follows that $p\colon \wt{Ge} \lra \ol{Ge}$ is birational as well. Therefore to conclude the proof of Theorem~\ref{teo:normalizzazione} we are left to show that $\widetilde A(Ge)$ is integral over $A(Ge)$. The argument of the following proof is due to M.~Brion (see \cite[Proposition 2.1]{CCM}).

\begin{proposition}
$\widetilde A(Ge)$ is an integral extension of $A(Ge)$.
\end{proposition}

\begin{proof}
Denote $Z = \mP(V(\grl_1)) \times \ldots \times \mP(V(\grl_m))$, for $i=1, \ldots, m$ let $p_i \colon Z \lra \mP(V(\grl_i))$ be the projection. Denote $\calA_Z = \calA(p_1^*\calO(1),\ldots,p_m^*\calO(1))$, a sheaf of $\calO_Z$-algebras, and set $L_Z = \Spec \grG(Z,\calA_Z)$. Similarly denote $\calA_X = \calA(\calL_1,\ldots,\calL_m)$, a sheaf of $\calO_X$-algebras, and set $L_X = \Spec \grG(X,\calA_X)$. Then we have a pullback diagram
\[
    \xymatrix
    {
        L_X \ar@{->}[rr]^{\ol\phi} \ar@{->}[d] & & L_Z \ar@{->}[d] \\
        X  \ar@{->}[rr]^{\phi} & & Z 
    }
\]
Notice that $\wt A(Ge) = \grG(L_X,\calO_{L_X}) = \grG(L_Z,\ol\phi_* \calO_{L_X})$, whereas $A(Ge)$ is the image of the natural morphism $\grG(L_Z,\calO_{L_Z}) \lra \grG(L_Z,\ol\phi_* \calO_{L_X})$. Notice that $\ol\phi$ is projective, so that $\ol\phi_* \calO_{L_X}$ is a coherent sheaf on $L_Z$. Hence $\grG(L_Z,\ol\phi_* \calO_{L_X})$ is a finitely generated $\grG(L_Z,\calO_{L_Z})$-module, or equivalently $\wt A(Ge)$ is a finitely generated $A(Ge)$-module.
\end{proof}

\section{The spherical systems}\label{s:1}

For the notation and the generalities about wonderful subgroups and Luna spherical systems we refer to our previous paper \cite[Section 1]{BCG}.

In this section we compute the wonderful subgroups of $K$ associated with the spherical systems given in the tables of Appendix~\ref{B}. Since, as one can check case-by-case, they correspond up to conjugation to the normalizers of the centralizers $\mathrm N_K(K_e)$ of the representatives $e$ given in Appendix~\ref{A}, we obtain the following.

\begin{proposition} The Luna spherical systems of the wonderful $K$-varieties associated to the spherical nilpotent $K$-orbits in $\mathfrak p$, for the classical symmetric pairs of Hermitian type, are those given in Tables~7--11.
\end{proposition}

We start with the spherical systems whose set $\Sigma$ is empty: 
\begin{itemize}
\item Table~7: 1.1 ($r=1$), 1.2 ($r=1$), 1.3 ($r=s=1$), 1.6 ($r=s=0$ and $q=1$), 1.7 ($r=s=0$ and $p=1$);
\item Table~8: 2.1, 2.2; 
\item Table~9: 3.1 ($r=1$), 3.2 ($r=1$), 3.3 ($r=s=1$); 
\item Table~10: 4.1, 4.2; 
\item Table~11: 5.1 ($r=1$), 5.2 ($r=1$), 5.3 ($r=s=1$).
\end{itemize}
Here it is immediate to see that the parabolic subgroup $Q$ of $K$ given in Appendix~\ref{A} is the wonderful subgroup associated with the corresponding spherical $K$-system. 

For the remaining spherical systems, let $M$ be the Levi subgroup of $K$ corresponding with $\supp\Sigma$. We will just compute the wonderful subgroup $H$ of $M$ associated with the $M$-spherical system obtained from the given $K$-spherical system by localization in $\supp\Sigma$. It is then immediate to deduce from $H$ which is the wonderful subgroup $K$ associated to the given $K$-spherical system and check that it is equal, up to conjugation, to the normalizer of $K_e$ given in the corresponding case of Appendix~\ref{A}.

\subsection{Symmetric cases}\label{ss:symmetric cases}

From the following $K$-spherical systems, by localizing in $\supp\Sigma$, we get the $M$-spherical systems of symmetric subgroups of $M$. Since these are well-known, we just provide a reference:
\begin{itemize}
\item in the cases 1.1 ($r>1$), 1.2 ($r>1$), 1.3 ($r>1$ or $s>1$) of Table~7 we get the case 2 of \cite{BP15}, or the direct product of two of them; 
\item in the cases 1.6 ($r=s=0$ and $q>1$), 1.7 ($r=s=0$ and $p>1$) of Table~7, 2.4 of Table~8, 4.4 of Table~10, 5.4 of Table~11 we get the case 3 of \cite{BP15};
\item in the case 2.3 of Table~8 we get the case 9 of \cite{BP15};
\item in the cases 3.1 ($r>1$), 3.2 ($r>1$), 3.3 ($r>1$ or $s>1$) of Table~9 we get the case 5 of \cite{BP15}, or the direct product of two of them;
\item in the case 4.3 of Table~10 we get the case 15 of \cite{BP15};
\item in the cases 5.1 ($r>1$), 5.2 ($r>1$), 5.3 ($r>1$ or $s>1$) of Table~11 we get the case 6 of \cite{BP15}, or the direct product of two of them.
\end{itemize}

\subsection{Other reductive cases}\label{ss:reductive}

From the $K$-spherical systems 1.4 and 1.5 of Table~7, by localizing in $\supp\Sigma$, we get the $M$-spherical system of a wonderful reductive (but not symmetric) subgroup of $M$: the case 43 of \cite{BP15}.

\subsection{Morphisms of type $\mathcal L$}\label{ss:typeL}

In the remaining cases, 1.6 ($r+s>0$) and 1.7 ($r+s>0$) of Table~7, the spherical $K$-system $(S^\mathrm p,\Sigma,\mathrm A)$ admits a distinguished set of colors $\Delta'$ such that the corresponding quotient 
\[(S^\mathrm p/\Delta',\Sigma/\Delta',\mathrm A/\Delta')\] 
is the spherical system of a wonderful $K$-variety which is obtained by parabolic induction from a wonderful $K_h$-variety. 
Such distinguished set of colors $\Delta'$ may be not minimal, let us describe it in the case~1.6, the other one is similar.
We assume $r$ and $s$ both to be non-zero and $r+s<q-1$, in the other cases the description is similar but simpler. 

Under this assumption, localizing the $K$-spherical system of the case~1.6 in $\supp\Sigma$, 
we obtain the following spherical system,
which we label as $\mathsf{a^y}(r,r)+\mathsf a(t)+\mathsf{a^y}(s,s)$ (here $t=q-r-s-1$)
for a group $M$ of semisimple type $\mathsf A_r\times\mathsf A_{r+s+t}\times\mathsf A_s$:
\begin{itemize}
\item[]$S^\mathrm p=\{\alpha'_{r+2},\ldots,\alpha'_{r+t-1}\}$,
\item[]$\Sigma=\{\alpha_1,\ldots,\alpha_r,\ \alpha'_1,\ldots,\alpha'_r,\ \alpha'_{r+1}+\ldots+\alpha'_{r+t},\ \alpha'_{r+t+1},\ldots,\alpha'_{r+t+s},\ \alpha''_1,\ldots,\alpha''_s)\}$,
\item[]$\Delta=\{D_1,\ldots,D_{2r+1},\ D_{2r+2},D_{2r+3},\ D_{2r+4},\ldots,D_{2r+2s+4}\}$,
\item[] and full Cartan pairing
\begin{itemize}
\item[]$\alpha_1=D_1+D_2-D_3$,
\item[]$\alpha_i=-D_{2i-2}+D_{2i-1}+D_{2i}-D_{2i+1}$ for $2\leq i\leq r$,
\item[]$\alpha'_i=-D_{2i-1}+D_{2i}+D_{2i+1}-D_{2i+2}$ for $1\leq i\leq r$,
\item[]$\alpha'_{r+1}+\ldots+\alpha'_{r+t}=-D_{2r+1}+D_{2r+2}+D_{2r+3}-D_{2r+4}$,
\item[]$\alpha'_{r+t+i}=-D_{2r+2i+1}+D_{2r+2i+2}+D_{2r+2i+3}-D_{2r+2i+4}$ for $1\leq i\leq s$,
\item[]$\alpha''_i=-D_{2r+2i+2}+D_{2r+2i+3}+D_{2r+2i+4}-D_{2r+2i+5}$ for $1\leq i\leq s-1$,
\item[]$\alpha''_s=-D_{2r+2s+2}+D_{2r+2s+3}+D_{2r+2s+4}$.
\end{itemize}
\end{itemize}
Notice that we have re-enumerated the simple roots of $M$ to simplify the notation.

If $t>1$ the Luna diagram is as follows.
\[\begin{picture}(25800,3300)(-300,-1500)
\multiput(0,0)(15300,0){2}{
\multiput(0,0)(8100,0){2}{\usebox{\edge}}
\multiput(1800,0)(4500,0){2}{\usebox{\shortsusp}}
\multiput(0,0)(6300,0){2}{
\multiput(0,0)(1800,0){3}{\usebox{\aone}}}
\multiput(0,1800)(6300,0){2}{\line(0,-1){900}}
\put(0,1800){\line(1,0){6300}}
\multiput(1800,1500)(6300,0){2}{\line(0,-1){600}}
\put(1800,1500){\line(1,0){3150}}
\multiput(4950,1500)(1650,0){2}{\multiput(150,0)(300,0){3}{\line(1,0){150}}}
\put(7650,1500){\line(1,0){450}}
\multiput(3600,1200)(6300,0){2}{\line(0,-1){300}}
\put(3600,1200){\line(1,0){2600}}
\put(6400,1200){\line(1,0){1600}}
\put(8200,1200){\line(1,0){1700}}
\multiput(1800,-1500)(4500,0){2}{\line(0,1){600}}
\put(1800,-1500){\line(1,0){4500}}
\multiput(3600,-1200)(4500,0){2}{\line(0,1){300}}
\put(3600,-1200){\line(1,0){1350}}
\multiput(4950,-1200)(1650,0){2}{\multiput(150,0)(300,0){3}{\line(1,0){150}}}
\put(7650,-1200){\line(1,0){450}}
\multiput(0,600)(1800,0){2}{\usebox{\toe}}
\multiput(8100,600)(1800,0){2}{\usebox{\tow}}
}
\multiput(9900,0)(3600,0){2}{\usebox{\edge}}
\put(11700,0){\usebox{\shortsusp}}
\put(11400,-300){\multiput(300,300)(1800,0){2}{\circle{600}}\multiput(300,300)(25,25){13}{\circle*{70}}\multiput(600,600)(300,0){4}{\multiput(0,0)(25,-25){7}{\circle*{70}}}\multiput(750,450)(300,0){4}{\multiput(0,0)(25,25){7}{\circle*{70}}}\multiput(2100,300)(-25,25){13}{\circle*{70}}}
\end{picture}\]

Let us consider the following subsets of colors:
\begin{itemize}
\item[]$\Delta_1=\{D_{2i}\,:\,1\leq i\leq r\}$,
\item[]$\Delta_2=\{D_{2r+2i+3}\,:\,1\leq i\leq s\}$.
\end{itemize}
Notice that $\Delta_1$ (resp.\ $\Delta_2$) consists of the colors represented by circles above vertices on the left (resp.\ right) hand side of the diagram.
Both are minimal distinguished with quotient of higher defect, see \cite[Section~1.5.5]{BCG}.
The quotient by $\Delta'=\Delta_1\cup\Delta_2$ is as follows.
\[\Sigma/\Delta'=\{\alpha_2+\alpha'_1,\ldots,\alpha_r+\alpha'_{r-1},\ \alpha'_{r+1}+\ldots+\alpha'_{r+t},\ 
\alpha'_{r+t+2}+\alpha''_{1},\ldots,\alpha'_{r+t+s}+\alpha''_{s-1}\}\]
\[\begin{picture}(25800,1950)(-300,-1050)
\multiput(0,0)(15300,0){2}{
\multiput(0,0)(8100,0){2}{\usebox{\edge}}
\multiput(1800,0)(4500,0){2}{\usebox{\shortsusp}}
\multiput(0,0)(6300,0){2}{
\multiput(0,0)(1800,0){3}{\usebox{\wcircle}}}
\multiput(1800,-1050)(4500,0){2}{\line(0,1){750}}
\put(1800,-1050){\line(1,0){4500}}
\multiput(3600,-750)(4500,0){2}{\line(0,1){450}}
\put(3600,-750){\line(1,0){2600}}
\put(6400,-750){\line(1,0){1700}}
}
\multiput(9900,0)(3600,0){2}{\usebox{\edge}}
\put(11700,0){\usebox{\shortsusp}}
\put(11400,-300){\multiput(300,300)(1800,0){2}{\circle{600}}\multiput(300,300)(25,25){13}{\circle*{70}}\multiput(600,600)(300,0){4}{\multiput(0,0)(25,-25){7}{\circle*{70}}}\multiput(750,450)(300,0){4}{\multiput(0,0)(25,25){7}{\circle*{70}}}\multiput(2100,300)(-25,25){13}{\circle*{70}}}
\end{picture}\]

This spherical system corresponds to a subgroup of $M$ which is a parabolic induction of a symmetric subgroup, that is, it can be decomposed as $LP^\mathrm{u}$ (a Levi decomposition) where $P=L_PP^\mathrm{u}$ is the parabolic subgroup of $M$ corresponding to $\supp(\Sigma/\Delta')$ and $L$ is a symmetric subgroup of $L_P$ corresponding to
\[\begin{picture}(25800,1950)(-300,-1050)
\multiput(0,0)(15300,0){2}{
\multiput(1800,0)(4500,0){2}{\usebox{\shortsusp}}
\multiput(0,0)(4500,0){2}{
\multiput(1800,0)(1800,0){2}{\usebox{\wcircle}}}
\multiput(1800,-1050)(4500,0){2}{\line(0,1){750}}
\put(1800,-1050){\line(1,0){4500}}
\multiput(3600,-750)(4500,0){2}{\line(0,1){450}}
\put(3600,-750){\line(1,0){2600}}
\put(6400,-750){\line(1,0){1700}}
}
\put(11700,0){\usebox{\shortsusp}}
\put(11400,-300){\multiput(300,300)(1800,0){2}{\circle{600}}\multiput(300,300)(25,25){13}{\circle*{70}}\multiput(600,600)(300,0){4}{\multiput(0,0)(25,-25){7}{\circle*{70}}}\multiput(750,450)(300,0){4}{\multiput(0,0)(25,25){7}{\circle*{70}}}\multiput(2100,300)(-25,25){13}{\circle*{70}}}
\end{picture}\]
More explicitely, $L$ is the normalizer of $\mathrm{SL}(r)\times\mathrm{GL}(t)\times\mathrm{SL}(s)$ in $L_P$ of semisimple type $\mathsf A_r\times\mathsf A_r\times\mathsf A_t\times\mathsf A_s\times\mathsf A_s$, where $\mathrm{SL}(r)$ is skew-diagonal in $\mathrm{SL}(r)\times\mathrm{SL}(r)$ and $\mathrm{SL}(s)$ is skew-diagonal in $\mathrm{SL}(s)\times\mathrm{SL}(s)$.

The wonderful subgroup $H$ associated with the above spherical $M$-system $\mathsf{a^y}(r,r)+\mathsf a(t)+\mathsf{a^y}(s,s)$ can therefore be taken as $L_HH^\mathrm u$, with $L_{H}\subset L$ and $H^\mathrm u\subset P^\mathrm u$, where $\mathrm{Lie}\,P^\mathrm u/\mathrm{Lie}\,H^\mathrm u$ is the direct sum of two simple $L_{H}$-modules while $L_{H}$ and $L$ differ only by their connected center. 

The codimension of $L_{H}$ in $L$ is equal to the increase in defect, which in this case is equal to 2.

The unipotent radical $\mathrm{Lie}\,H^\mathrm u$ is as follows. 
There exist $W_{0,1}$ and $W_{1,1}$ $L_{H}$-submodules of $\mathrm{Lie}\,P^\mathrm u$ of dimension $r$, isomorphic as $L_{H}$-modules but not as $L$-modules. Analogously there exist $W_{0,2}$ and $W_{1,2}$ $L_{H}$-submodules of $\mathrm{Lie}\,P^\mathrm u$ of dimension $s$, isomorphic as $L_{H}$-modules but not as $L$-modules. Denoting by $V$ the $L_{H}$-complement of $W_{0,1}\oplus W_{1,1}\oplus W_{0,2} \oplus W_{1,2}$ in $\mathrm{Lie}\,P^\mathrm u$, as $L_{H}$-module we have
\[\mathrm{Lie}\,H_1^\mathrm u=W_1\oplus W_2\oplus V\]
where $W_1$ is a simple $L_{H}$-submodule of $W_{0,1}\oplus W_{1,1}$ which projects non-trivially on both summands,
and analogously $W_2$ is a simple $L_{H}$-submodule of $W_{0,2}\oplus W_{1,2}$ which projects non-trivially on both summands. 

\begin{remark}
Notice that the given subsets of colors $\Delta_1$ and $\Delta_2$ decompose the above spherical system in the sense of \cite[Definition 2.2.2]{BP16}, hence the corresponding wonderful variety is a non-trivial wonderful fiber product. 
\end{remark}

\section{Projective normality}\label{s:2}

In this section we prove the following result, that will be used to study the normality of 
the spherical nilpotent $K$-orbit closures in $\gop$. Reasoning as in Section \ref{s:3}, recall that every spherical nilpotent $K$-orbit in $\gop$ determines naturally a wonderful $K$-variety. 

\begin{theorem}	\label{teo: projnorm}
Let $(\mathfrak g,\mathfrak k)$ be a classical symmetric pair of Hermitian type, let $\calO \subset \gop$ be a spherical nilpotent $K$-orbit and let $X$ be the wonderful $K$-variety associated to $\calO$. Then the multiplication of sections $m_{\calL,\calL'}$ is surjective for all globally generated line bundles $\calL, \calL' \in \Pic(X)$.
\end{theorem}

Some generalities on line bundles on wonderful varieties and their sections have already been recalled in Section \ref{s:3}, we keep on using the notation introduced in \cite[Section 1]{BGM}.

\subsection{General reductions}\label{ss:General reductions}

The surjectivity of the multiplication of sections of globally generated line bundles on a wonderful variety is known for some important families, namely for the \textit{wonderful symmetric varieties} (see \cite{CM_projective-normality}), for the \textit{wonderful model varieties} (of simply-connected type) and for the \textit{wonderful comodel varieties} (see \cite{BGM}).

We will essentially reduce the proof of Theorem \ref{teo: projnorm} to the known cases recalled above, by making use of the operations of localization, quotient and parabolic induction on spherical systems, which provide us with some reduction steps in the study of the multiplication maps. We recall such reductions, which are proved in \cite{BCG} and \cite{BGM}, and then we will apply them to the cases under consideration.

\begin{lemma}[{\cite[Lemma 2.4]{BCG}}]	\label{lem: projnorm localizzazioni}
Let $X$ be a wonderful variety and let $X' \subset X$ be a wonderful subvariety. If $m_{\calL, \calL'}$ is surjective for all globally generated $\calL, \calL' \in \Pic(X)$, then $m_{\calL,\calL'}$ is surjective for all globally generated $\calL, \calL' \in \Pic(X')$.
\end{lemma}

\begin{lemma}[{\cite[Corollary 1.4]{BGM}}]	\label{lem: projnorm quotients}
Let $X$ be a wonderful variety with set of colors $\grD$, let $X'$ be a quotient of $X$ by a distinguished subset $\grD_0 \subset \grD$ with set of colors $\grD'$ and identify $\grD'$ with $\grD \senza \grD_0$. If $D \in \mN \grD$ and $\supp(D) \cap \grD_0 = \vuoto$ and if $\calL_D \in \Pic(X)$ and $\calL'_D \in \Pic(X')$ are the line bundles corresponding to $D$ regarded as an element in $\mN \grD$ and in $\mN\grD'$, then $\grG(X, \calL_D) = \grG(X', \calL'_D)$.

In particular, if $m_{D,E}$ is surjective for all $D,E \in \mN\grD$, then  $m_{D', E'}$ is surjective for all $D', E' \in \mN \grD'$.
\end{lemma}

\begin{lemma}[{\cite[Proposition 1.6]{BGM}}] \label{lem: projnorm parabolic induction}
Let $X$ be a wonderful variety and suppose that $X$ is the parabolic induction of a wonderful variety $X'$. Then for all $\calL,\calL'$ in $\Pic(X)$ the multiplication $m_{\calL, \calL'}$ is surjective if and only if the multiplication $m_{\calL|_{X'}, \calL'|_{X'}}$ is surjective.
\end{lemma}

We now apply previous reductions to our cases.
In particular, we will show that to prove Theorem~\ref{teo: projnorm} 
it is enough to prove the surjectivity of the multiplication just for the
following basic case, labelled $\mathsf {a^x}(1,1,1)$, since the other basic cases have already been treated in \cite{CM_projective-normality} and \cite{BGM}:

\[\begin{picture}(3600,2000)(-300,-1500)
\multiput(0,0)(1800,0){3}{\usebox{\aone}}
\multiput(0,-900)(3600,0){2}{\line(0,-1){300}}
\put(0,-1200){\line(1,0){3600}}
\multiput(0,900)(1800,0){2}{\line(0,1){300}}
\put(0,1200){\line(1,0){1800}}
\put(2700,600){\line(1,0){650}}
\put(2050,-600){\line(1,0){650}}
\put(2700,-600){\line(0,1){1200}}
\end{picture}
\]

Let $\calO \subset \gop$ be a spherical nilpotent $K$-orbit as in Theorem \ref{teo: projnorm} and let $X$ be the corresponding wonderful variety. When $X$ is a flag variety, or equivalently $\Sigma=\vuoto$, the surjectivity of the multiplication is trivial.

By Lemma~\ref{lem: projnorm parabolic induction}, 
the surjectivity of the multiplication on
$X$ is reduced to the surjectivity of the multiplication on $Z$, the localization of $X$ at
the subset $\supp \grS\subset S$. These localizations are described in Section~\ref{s:1}.

In the cases 
\begin{itemize}
\item[] 1.6 ($r=s=0$ and $q>1$), 1.7 ($r=s=0$ and $p>1$), 
\item[] 2.3, 2.4, 
\item[] 4.3, 4.4, 
\item[] 5.4,
\end{itemize}
treated in Section~\ref{ss:symmetric cases}, the wonderful variety
$Z$ is a rank one wonderful variety which is homogeneous under its automorphism group (see \cite{Akh} for a description of these varieties).
Therefore in these cases $Z$ is a flag variety, and the surjectivity of the multiplication is trivial.

In the remaining cases of Section~\ref{ss:symmetric cases} the wonderful
variety $Z$ is the wonderful compactification of an adjoint symmetric
variety, and the surjectivity of the multiplication holds thanks to \cite{CM_projective-normality}.

In the cases 1.4 and 1.5, treated in Section~\ref{ss:reductive}, we get the wonderful variety $Z$ with spherical system $\mathsf{a^x}(1,1,1)$.

In the cases of Section~\ref{ss:typeL}, $Z$ is the wonderful
variety with spherical system $\mathsf{a^y}(r,r)+\mathsf a(t)+\mathsf{a^y}(s,s)$. 
The surjectivity of the multiplication in this case can be
reduced to the surjectivity of the multiplication for a comodel wonderful variety, which is known by
\cite[Theorem~5.2]{BGM}.

Let indeed $Z'$ be the comodel wonderful
variety of cotype $\sfA_{2(r+t+s)}$, which is the wonderful variety with
the following spherical system for a group $G$ of semisimple type $\sfA_{r+t+s+1} \times \sfA_{r+t+s}$, where the set $\Sigma$ of spherical roots is equal to the set of simpe roots of $G$. 
\[\begin{picture}(30300,6000)(-300,-3000)
\multiput(0,0)(10800,0){2}{\multiput(0,0)(1800,0){2}{\usebox{\edge}}}
\put(3600,0){\thicklines\multiput(0,0)(6100,0){2}{\line(1,0){1100}}\multiput(1300,0)(400,0){12}{\line(1,0){200}}}
\multiput(17100,0)(9000,0){2}{\multiput(0,0)(1800,0){2}{\usebox{\edge}}}
\put(20700,0){\thicklines\multiput(0,0)(4300,0){2}{\line(1,0){1100}}\multiput(1300,0)(400,0){8}{\line(1,0){200}}}
\multiput(0,0)(10800,0){2}{\multiput(0,0)(1800,0){3}{\usebox{\aone}}}
\multiput(17100,0)(9000,0){2}{\multiput(0,0)(1800,0){3}{\usebox{\aone}}}
\multiput(0,3000)(29700,0){2}{\line(0,-1){2100}}
\put(0,3000){\line(1,0){29700}}
\multiput(1800,2700)(26100,0){2}{\line(0,-1){1800}}
\put(1800,2700){\line(1,0){26100}}
\multiput(3600,2400)(22500,0){2}{\line(0,-1){1500}}
\put(3600,2400){\line(1,0){22500}}
\multiput(10800,-1800)(9900,0){2}{\line(0,1){900}}
\put(10800,-1800){\line(1,0){9900}}
\multiput(12600,-1500)(6300,0){2}{\line(0,1){600}}
\put(12600,-1500){\line(1,0){6300}}
\multiput(14400,-1200)(2700,0){2}{\line(0,1){300}}
\put(14400,-1200){\line(1,0){2700}}
\multiput(1800,-3000)(27900,0){2}{\line(0,1){2100}}
\put(1800,-3000){\line(1,0){27900}}
\multiput(3600,-2700)(24300,0){2}{\line(0,1){1800}}
\put(3600,-2700){\line(1,0){24300}}
\multiput(5400,-2400)(0,300){4}{\line(0,1){150}}
\put(5400,-2400){\line(1,0){20700}}
\put(26100,-2400){\line(0,1){1500}}
\multiput(9000,1800)(0,-300){3}{\line(0,-1){150}}
\put(9000,1800){\line(1,0){11700}}
\put(20700,1800){\line(0,-1){900}}
\multiput(10800,1500)(8100,0){2}{\line(0,-1){600}}
\put(10800,1500){\line(1,0){8100}}
\multiput(12600,1200)(4500,0){2}{\line(0,-1){300}}
\put(12600,1200){\line(1,0){4500}}
\multiput(0,600)(1800,0){3}{\usebox{\toe}}
\multiput(10800,600)(1800,0){2}{\usebox{\toe}}
\multiput(17100,600)(1800,0){3}{\usebox{\toe}}
\multiput(26100,600)(1800,0){2}{\usebox{\toe}}
\end{picture}\]
Consider the wonderful subvariety of $Z'$ associated to
$\Sigma \smallsetminus \{\alpha_{r+1}, \ldots, \alpha_{r+t+1}\}$, 
\[\begin{picture}(30300,5400)(-300,-2700)
\multiput(0,0)(9000,0){2}{\multiput(0,0)(3600,0){2}{\usebox{\edge}}}
\multiput(1800,0)(9000,0){2}{\usebox{\shortsusp}}
\put(5400,0){\usebox{\susp}}
\multiput(17100,0)(5400,0){3}{\usebox{\shortsusp}}
\multiput(18900,0)(5400,0){2}{\usebox{\dynkinathree}}
\multiput(0,0)(10800,0){2}{\multiput(0,0)(1800,0){3}{\usebox{\aone}}}
\multiput(5400,0)(3600,0){2}{\usebox{\wcircle}}
\multiput(17100,0)(7200,0){2}{\multiput(0,0)(1800,0){4}{\usebox{\aone}}}
\multiput(0,2700)(29700,0){2}{\line(0,-1){1800}}
\put(0,2700){\line(1,0){29700}}
\multiput(1800,2400)(26100,0){2}{\line(0,-1){1500}}
\multiput(1800,2400)(14400,0){2}{\line(1,0){11700}}
\multiput(13500,2400)(1650,0){2}{\multiput(150,0)(300,0){3}{\line(1,0){150}}}
\multiput(3600,2100)(22500,0){2}{\line(0,-1){1200}}
\put(3600,2100){\line(1,0){22500}}
\multiput(10800,-1800)(9900,0){2}{\line(0,1){900}}
\put(10800,-1800){\line(1,0){9900}}
\multiput(12600,-1500)(6300,0){2}{\line(0,1){600}}
\multiput(12600,-1500)(4500,0){2}{\line(1,0){1800}}
\multiput(14400,-1500)(1650,0){2}{\multiput(150,0)(300,0){3}{\line(1,0){150}}}
\multiput(14400,-1200)(2700,0){2}{\line(0,1){300}}
\put(14400,-1200){\line(1,0){2700}}
\multiput(1800,-2700)(27900,0){2}{\line(0,1){1800}}
\put(1800,-2700){\line(1,0){27900}}
\multiput(3600,-2400)(24300,0){2}{\line(0,1){1500}}
\put(3600,-2400){\line(1,0){24300}}
\multiput(10800,1500)(8100,0){2}{\line(0,-1){600}}
\put(10800,1500){\line(1,0){8100}}
\multiput(12600,1200)(4500,0){2}{\line(0,-1){300}}
\put(12600,1200){\line(1,0){4500}}
\multiput(0,600)(1800,0){2}{\usebox{\toe}}
\multiput(10800,600)(1800,0){2}{\usebox{\toe}}
\multiput(17100,600)(1800,0){7}{\usebox{\toe}}
\end{picture}\]
then the set of
colors $\{D_{\alpha'_{s+1}}^+, D_{\alpha'_{s+2}}^\pm, \ldots, D_{\alpha'_{s+t-1}}^\pm, D_{\alpha'_{s+t}}^-\}$ is distinguished, and the corresponding
quotient 
is a parabolic induction of $Z$. 
\[\begin{picture}(30300,5400)(-300,-2700)
\multiput(0,0)(9000,0){2}{\multiput(0,0)(3600,0){2}{\usebox{\edge}}}
\multiput(1800,0)(9000,0){2}{\usebox{\shortsusp}}
\put(5400,0){\usebox{\susp}}
\multiput(17100,0)(5400,0){3}{\usebox{\shortsusp}}
\multiput(18900,0)(5400,0){2}{\usebox{\dynkinathree}}
\multiput(0,0)(10800,0){2}{\multiput(0,0)(1800,0){3}{\usebox{\aone}}}
\multiput(5400,0)(3600,0){2}{\usebox{\wcircle}}
\multiput(17100,0)(9000,0){2}{\multiput(0,0)(1800,0){3}{\usebox{\aone}}}
\put(22200,-300){\multiput(300,300)(1800,0){2}{\circle{600}}\multiput(300,300)(25,25){13}{\circle*{70}}\multiput(600,600)(300,0){4}{\multiput(0,0)(25,-25){7}{\circle*{70}}}\multiput(750,450)(300,0){4}{\multiput(0,0)(25,25){7}{\circle*{70}}}\multiput(2100,300)(-25,25){13}{\circle*{70}}}
\multiput(0,2700)(29700,0){2}{\line(0,-1){1800}}
\put(0,2700){\line(1,0){29700}}
\multiput(1800,2400)(26100,0){2}{\line(0,-1){1500}}
\multiput(1800,2400)(14400,0){2}{\line(1,0){11700}}
\multiput(13500,2400)(1650,0){2}{\multiput(150,0)(300,0){3}{\line(1,0){150}}}
\multiput(3600,2100)(22500,0){2}{\line(0,-1){1200}}
\put(3600,2100){\line(1,0){22500}}
\multiput(10800,-1800)(9900,0){2}{\line(0,1){900}}
\put(10800,-1800){\line(1,0){9900}}
\multiput(12600,-1500)(6300,0){2}{\line(0,1){600}}
\multiput(12600,-1500)(4500,0){2}{\line(1,0){1800}}
\multiput(14400,-1500)(1650,0){2}{\multiput(150,0)(300,0){3}{\line(1,0){150}}}
\multiput(14400,-1200)(2700,0){2}{\line(0,1){300}}
\put(14400,-1200){\line(1,0){2700}}
\multiput(1800,-2700)(27900,0){2}{\line(0,1){1800}}
\put(1800,-2700){\line(1,0){27900}}
\multiput(3600,-2400)(24300,0){2}{\line(0,1){1500}}
\put(3600,-2400){\line(1,0){24300}}
\multiput(10800,1500)(8100,0){2}{\line(0,-1){600}}
\put(10800,1500){\line(1,0){8100}}
\multiput(12600,1200)(4500,0){2}{\line(0,-1){300}}
\put(12600,1200){\line(1,0){4500}}
\multiput(0,600)(1800,0){2}{\usebox{\toe}}
\multiput(10800,600)(1800,0){2}{\usebox{\toe}}
\multiput(17100,600)(1800,0){3}{\usebox{\toe}}
\multiput(26100,600)(1800,0){2}{\usebox{\toe}}
\end{picture}\]
Therefore the surjectivity of the multiplication of $Z$ follows from that of $Z'$ thanks to 
Lemmas~\ref{lem: projnorm localizzazioni},~\ref{lem: projnorm quotients} and~\ref{lem: projnorm parabolic induction}.

\subsection{Projective normality of $\mathsf{a^x}(1,1,1)$} \label{ss:ax111} 

Consider $X$ for the semisimple group $G = \mathrm{SL}(2) \times \mathrm{SL}(2) \times \mathrm{SL}(2)$ with spherical system $\mathsf{a^x}(1,1,1)$.

The spherical roots of $X$ coincide with the simple roots, and we enumerate them as follows: $\grs_1 = \gra$, $\grs_2 = \gra'$, $\grs_3 = \gra''$. Moreover we enumerate the colors of $X$ in the following way: $D_1 = D_{\gra''}^+$, $D_2 = D_{\gra}^-$, $D_3 = D_{\gra}^+$. Then the spherical roots are expressed in terms of colors as follows: 
\[\begin{array}{rr}
\grs_1 = & -D_1 + D_2 +D_3, \\ 
\grs_2 = & D_1 - D_2 +D_3, \\
\grs_3 = & D_1 + D_2 -D_3.
\end{array}\]

Notice that the restriction $\gro : \Pic(X) \ra \grL$ is injective in this case (see e.g.\ \cite[Lemma~30.24]{Ti}). Indeed, by identifying $\grL^+$ with $\mN^3$ we have
$$\begin{array}{ccc}
\gro(\calL_{D_1}) = (0,1,1), & 
\gro(\calL_{D_2}) = (1,0,1), &
\gro(\calL_{D_3}) = (1,1,0),\\
\gro(\calL_{\grs_1}) = (2,0,0), &
\gro(\calL_{\grs_2}) = (0,2,0), &
\gro(\calL_{\grs_3}) = (0,0,2).\\
\end{array}
$$
Thus every $\calL \in \Pic(X)$ is uniquely determined by the corresponding triple $\gro(\calL) \in \mZ^3$. 

The surjectivity of the multiplication of sections of globally generated line bundles can be proved fairly easily in this case by making use of the machinery of low triples developed in \cite{BGM}, namely by classifying the covering differences of the partial ordered set $(\mN \grD, \leq_\grS)$ and then studying the fundamental low triples. However, in order to increase the accessibility to the non-experts, we will write down a proof by more elementary language, only involving representation theoretic considerations. In any case, even though we do not mention low triples and covering differences, they are intrinsically hidden inside the arguments that appear in the present proof.

Let $H \subset G$ be the generic stabilizer of $X$. Identifying sections of line bundles on $X$ with functions on $G$ which are semiinvariant with respect to the right action of $H$, for every globally generated line bundle $\calL \in \Pic(X)$ we can regard the space of sections $\grG(X,\calL)$ as a $G$-stable submodule of $\mC[G]^{(H)}$.

If $m \in \mN$, denote by $V(m)$ the simple $\mathrm{SL(2)}$-module of highest weight $m$, and if $\ul m = (m,m',m'') \in \mN^3$ we denote $V(\ul m) = V(m) \otimes V(m') \otimes V(m'')$, regarded as a $G$-module.

Let
$$
	\mathrm T = \{(m, m', m'') \in \mN^3 \; : \; V(m'') \subset V(m) \otimes V(m')\}
$$
be the tensor semigroup of $\mathrm{SL}(2)$. Then by the Clebsch-Gordan rule we have
$$
	\mathrm T = \{(m, m', m'') \in \mN^3 \; : \; m+m'+m'' \in 2\mN \quad \text{and} \quad |m-m'| \leq m'' \leq m+m'\}.
$$
Notice that $\mathrm T$ is stable under permutations: indeed
$$
V(m'') \subset V(m) \otimes V(m') \iff \big(V(m)^* \otimes V(m')^* \otimes V(m'')\big)^{\diag \mathrm{SL(2)}} \neq 0,
$$
the latter conditon is independent on the order of the triple $(m,m',m'')$ since $V(m) \simeq V(m)^*$ and $V(m') \simeq V(m')^*$ are self-dual modules.

In our case, we can apply the previous discussion as follows. Indeed, by \cite[Case 43]{BP15} we can assume that
$$H = \big(Z_{\mathrm{SL(2)}} \times Z_{\mathrm{SL(2)}} \times Z_{\mathrm{SL(2)}}\big) \; \diag \mathrm{SL(2)}.$$
Since $V(\ul m) \simeq V(\ul m)^*$ is self dual, we have then
\begin{align*}
& \mC[G]^{(H)} = \mC[G]^{\diag \mathrm{SL(2)}} \simeq  \bigoplus_{\ul m \in \mN^3} V(\ul m) \otimes V(\ul m)^{\diag \mathrm{SL(2)}} = \bigoplus_{\ul m \in \mathrm T} V(\ul m).
\end{align*}

By the $G$-equivariant isomorphism $\mC[G]^{(H)} \simeq \bigoplus_{D \in \mN\grD} V_D$, as a consequence of the previous description we see that the map $\gro$ identifies the 
semigroup of globally generated line bundles on $X$ with the tensor semigroup $\mathrm T$.

The $G$-equivariant embedding $V(\ul m) \lra \mC[G]^{\diag \mathrm{SL(2)}}$ can be more explicitly expressed in terms of matrix coefficients as follows. Fix for all $k \in \mN$ an equivariant isomorphism $v \longmapsto \psi_v$ between $V(k)$ and $V(k)^*$.

If $\ul m = (m,m',m'') \in \mathrm T$, by the Clebsch-Gordan rule $V(m'')$ always occurs with multiplicity one in $V(m) \otimes V(m')$, thus we have an equivariant projection
$$
	\pi_{m,m'}^{m'\!'} \colon V(m) \otimes V(m') \lra V(m'')
$$
which is uniquely determined up to a scalar factor.

For any simple tensor $v \otimes v' \otimes v'' \in V(m) \otimes V(m') \otimes V(m'')$ we can define a matrix coefficient $f_{v \otimes v' \otimes v'\!' } \in \mC[G]$ by setting, for all $(g,g',g'') \in G$, 
$$
f_{v \otimes v' \otimes v'\!' } (g,g',g'') =  \langle \pi_{m,m'}^{m'\!'}(g v \otimes g' v'), g'' \psi_{v'\!'} \rangle,
$$
Notice that $f_{v \otimes v' \otimes v'\!'} \in \mC[G]^{\diag \mathrm{SL(2)}}$. Thus extending linearly we get an embedding of $V(\ul m)$ in $\mC[G]^{\diag \mathrm{SL(2)}}$.

Define a partial order on $\mathrm T$, by setting $(m,m',m'') \geq (n,n',n'')$ if the differences $m - n$, $m' - n'$, $m'' - n''$ are all nonnegative even numbers. This is the partial order on $\mathrm T$ induced by the partial order $\leq_\grS$ on $\mN \grD$ via the isomorphism $\mathrm T \simeq \mN \grD$.

For all $\ul m \in \mathrm T$, define a submodule of $\mC[G]^{\diag \mathrm{SL(2)}}$ by setting
$$
	\grG(\ul m) = \bigoplus_{\ul n \leq \ul m} \; V(\ul n).
$$

If $\ul m \in \mathrm T$ and $\calL_{\ul m} \in \Pic(X)$ is the corresponding globally generated line bundle on $X$, then by \eqref{eq:decomposition} we have a $G$-equivariant isomorphism
$$
	\grG(X,\calL_{\ul m}) \simeq \grG(\ul m).
$$
If moreover $\ul n \in \mathrm T$, then we have a commutative diagram 
\[
    \xymatrix
    {
        \grG(X,\calL_{\ul m}) \otimes \grG(X,\calL_{\ul n}) \ar[r] \ar_{\simeq}[d] & \grG(X,\calL_{\ul m+\ul n}) \ar^{\simeq}[d] \\ 
    	\grG(\ul m) \otimes \grG(\ul n) \ar[r] & \grG(\ul m+\ul n)
    }
\]
where the upper arrow is the multiplication of sections and the lower arrow is the multiplication in the invariant ring $\mC[G]^{\diag \mathrm{SL(2)}}$.

\begin{lemma}	\label{lemma:prodotti}
Let $\ul k = (k,k',k'')$, $\ul m = (m,m', m'')$, $\ul n = (n,n',n'') \in \mathrm T$. Suppose that $(m,n,k)$, $(m',n',k')$, $(m'',n'',k'')  \in \mathrm T$. If $V(k'')$ occurs with multiplicity one inside $V(m) \otimes V(m') \otimes V(n) \otimes V(n')$, then $V(\ul k) \subset V(\ul m) \cdot V(\ul n)$.
\end{lemma}

\begin{proof}
Consider the diagram
\[
    \xymatrix
    {
        V(m) \otimes V(n) \otimes V(m') \otimes V(n')  \ar^{\qquad \quad \; \pi_{m\!,n}^{k} \otimes \pi_{m'\!\!,n'}^{k'}}[rr] \ar_{\pi_{m\!,m'}^{m'\!'} \otimes \pi_{n\!,n'}^{n'\!'}}[d] & & V(k) \otimes V(k') \ar^{\pi_{k\!,k'}^{k'\!'}}[d] \\ 
    	V(m'') \otimes V(n'') \ar^{\pi_{m'\!'\!,n'\!'}^{k'\!'}}[rr] & & V(k'')
    }
\]
Then $\pi_{k,k'}^{k'\!'} \circ (\pi_{m,n}^{k} \otimes \pi_{m',n'}^{k'})$ and $\pi_{m'\!',n'\!'}^{k'\!'} \circ (\pi_{m,m'}^{m'\!'} \otimes \pi_{n,n'}^{n'\!'})$ are both non-zero and equivariant for the action of $\mathrm{SL}(2)$. Since $V(k'')$ occurs with multiplicity one inside $V(m) \otimes V(m') \otimes V(n) \otimes V(n')$ these two projections must be proportional, hence up to a renormalization we may assume that
$$\pi_{k,k'}^{k'\!'} \circ (\pi_{m,n}^{k} \otimes \pi_{m',n'}^{k'}) = \pi_{m'\!',n'\!'}^{k'\!'} \circ (\pi_{m,m'}^{m'\!'} \otimes \pi_{n,n'}^{n'\!'}).$$ 

Let now $v \otimes v' \otimes v'' \in V(\ul m)$ and $w \otimes w' \otimes w'' \in V(\ul n)$ be simple tensors and consider the product of the corresponding matrix coefficients.

Then we have
\begin{align*}
	\big(f_{v \otimes v' \otimes v'\!'}  f_{w \otimes w' \otimes w'\!' }\big) &  (g,g',g'') = \\
	&  = \langle (\pi_{m,m'}^{m'\!'} \otimes \pi_{n,n'}^{n'\!'}) \big( g.(v\otimes w) \otimes g'\!.(v' \otimes w') \big), \, g''\!.(\psi_{v'\!'} \otimes \psi_{w'\!'}) \rangle.
\end{align*}

Since $V(k'') \subset V(m'') \otimes V(n'')$, we can find a nonzero element $u'' = \sum_h v''_h \otimes w''_h \in V(k'')$, for suitable elements $v''_h \in V(m'')$ and $w''_h \in V(n'')$. Then, up to a nonzero scalar factor, we have
\begin{align*}
	\sum_h \big(f_{v \otimes v' \otimes v'\!'_h} & f_{w \otimes w' \otimes w'\!'_h}\big)  (g,g',g'') = \\
	&  = \langle (\pi_{m,m'}^{m'\!'} \otimes \pi_{n,n'}^{n'\!'}) \big( g(v\otimes w) \otimes g'\!.(v' \otimes w') \big), \, g''\!.(\sum_h \psi_{v'\!'_h} \otimes \psi_{w'\!'_h}) \rangle = \\
	&  = \langle \pi_{m'\!'\!,n'\!'}^{k'\!'} \circ (\pi_{m,m'}^{m'\!'} \otimes \pi_{n,n'}^{n'\!'}) \big( g.(v\otimes w) \otimes g'\!.(v' \otimes w') \big), \, g''\!.\psi_{u'\!'} \rangle = \\
	&  = \langle \pi_{k,k'}^{k'\!'} \circ (\pi_{m,n}^{k} \otimes \pi_{m',n'}^{k'}) \big( g.(v\otimes w) \otimes g'\!.(v' \otimes w') \big), \, g''\!.\psi_{u'\!'} \rangle = \\
	&  = \langle \pi_{k,k'}^{k'\!'} \big(g \pi_{m,n}^{k} (v\otimes w) \otimes g' \pi_{m',n'}^{k'} (v' \otimes w') \big), \, g''\psi_{u'\!'} \rangle,
\end{align*}
which is the matrix coefficient of
$$\pi_{m,n}^{k} (v\otimes w) \otimes \pi_{m',n'}^{k'}(v' \otimes w') \otimes u'' \in V(k) \otimes V(k') \otimes V(k'').$$
Since $V(k) \subset V(m) \otimes V(n)$ and $V(k') \subset V(m') \otimes V(n')$, we can find nonzero elements $u = \sum_i v_i \otimes w_i \in V(m) \otimes V(n)$ and $u' = \sum_j v'_j \otimes w'_j \in V(m') \otimes V(n')$ with $u \in V(k)$ and $u' \in V(k')$, thus
$$
	\sum_{i,j,h} \big(f_{v_i \otimes v'_j \otimes v'\!'_h} f_{w_i \otimes w'_j \otimes w'\!'_h}\big)  (g,g',g'') = \langle \pi_{k,k'}^{k'\!'} (g u \otimes g' u' ), \, g''\psi_{u'\!'} \rangle
$$
is the matrix coefficient of a nonzero element in $V(\ul k)$, and the claim follows.
\end{proof}

\begin{remark}
Notice that the multiplication in $\mC[G]^{\diag \mathrm{SL(2)}}$ is degenerate in the sense of \cite[Proposition 9.1]{BGM}, that is, the inclusion 
\[V(\ul m)\cdot V(\ul n)\subset\bigoplus_{\ul k\in\mathrm T\ :\ V(\ul k)\subset V(\ul m)\otimes V(\ul n)}V(\ul k)\]
is not necessarily an equality for all $\ul m$ and $\ul n$ in $\mathrm T$. For example, for $\ul k = (2,2,2)$ and $\ul m = (1,1,2)$, $V(\ul k)$ occurs in the tensor product $V(\ul m)^{\otimes 2}$. However $V(\ul k)$ does not occur in the symmetric product $\mathsf S^2 V(\ul m)$, thus $V(\ul k) \not \subset V(\ul m)^2$ inside $\mC[G]^{\diag \mathrm{SL(2)}}$.
\end{remark}

In terms of matrix coefficients, for $\ul k = (k,k',k'')$, $\ul m = (m,m', m'')$, $\ul n = (n,n',n'')$ in $\mathrm T$ with $V(\ul k)\subset V(\ul m)\otimes V(\ul n)$, it can be easily shown that
$$
V(\ul k) \subset V(\ul m) \cdot V(\ul n)
\quad \Longleftrightarrow \;
\pi_{m'\!'\!,n'\!'}^{k'\!'} \circ \big(\pi_{m\!,m'}^{m'\!'} \otimes \pi_{n\!,n'}^{n'\!'}\big) \circ \big(\iota^{m\!,n}_{k} \otimes \iota^{m'\!,n'}_{k'}\big)  \neq 0
$$
where $\iota^{m,n}_{k}$ and $\iota^{m',n'}_{k'}$ denote the equivariant injections $V(k) \lra V(m) \otimes V(n)$ and $V(k') \lra V(m') \otimes V(n')$, respectively.

%
%
%

\begin{proposition}
Let $\ul m$ and $\ul n \in \mathrm T$, then $\grG(\ul m) \cdot \grG(\ul n) = \grG(\ul m+\ul n)$.
\end{proposition}

\begin{proof}
We have to show that, for all $\ul k \leq \ul m + \ul n$, it holds $V(\ul k) \subset \grG(\ul m) \cdot \grG(\ul n)$. We proceed by induction on the sum $m+n+m'+n'+m''+n''$. Let $\ul k \in \mathrm T$ with $\ul k \leq \ul m + \ul n$.

Suppose first that $\ul k = \ul m + \ul n$. If $v_{\ul m} \in V(\ul m)$ and $v_{\ul n} \in V(\ul n)$ are highest weight vectors, then their product $v_{\ul m} v_{\ul n}$ is a highest weight vector of weight $\ul m + \ul n$, thus $V(\ul m + \ul n) \subset V(\ul m) \cdot V(\ul n) \subset \grG(\ul m) \cdot \grG(\ul n)$. Therefore we will assume that $\ul k < \ul m + \ul n$.

\textit{Case 1.} Suppose that both $\ul m$ and $\ul n$ have zero entries. Then $\ul m$ and $\ul n$ have a unique zero entry, and up to a permutation of the coordinates we may assume that $\ul m = (m,0,m)$ and $n'' > 0$. Thus we have either $\ul n = (n,0,n)$ or $\ul n = (0,n',n')$.

Suppose that $\ul n = (n,0,n)$. Then it must be $\ul k = (k,0,k)$ with $k < m+n$, thus $\ul k \leq \ul m +\ul n - (2,0,2)$. If $\ul m = \ul n = (1,0,1)$, then $\ul k = (0,0,0)$ and $V(\ul k) \subset V(\ul m) \cdot V(\ul n)$ by Lemma \ref{lemma:prodotti}. Therefore we can assume that either $m \geq 2$ or $n \geq 2$. Suppose that $m \geq 2$, the other case is similar: then $(m-2,0,m-2) \in \mathrm T$, and by induction
$$
	V(\ul k) \subset \grG(m-2,0,m-2) \cdot \grG(\ul n) \subset \grG(\ul m) \cdot \grG(\ul n).
$$

Suppose now that $\ul n = (0,n',n')$. Then $\ul k < (m,n',m+n')$, and since $k'' \leq k+k' \leq m+n'$ it must be $k'' < m+n'$. If $k+k' = m+n'$ or if $m=n' = 1$, then $\ul k = (m,n',k'')$ and by Lemma \ref{lemma:prodotti} we get $V(\ul k) \subset V(\ul m) \cdot V(\ul n)$. Thus we can assume that $k+k' < m+n'$, and that either $m \geq 2$ or $n' \geq 2$. Suppose we are in the first case, the other one is similar. Then $(m-2,0,m-2) \in \mathrm T$ and $\ul k \leq (m-2,m-2,0) + \ul n$, and the claim follows again by the induction.

\textit{Case 2.} Suppose that $k'' < m'' + n''$ and that either $m''>|m-m'|$ or $n''>|n-n'|$. We can assume that we are in the first case. Then $(m,m',m''-2) \in \mathrm T$ and $\ul k \leq (m,m',m''-2) + \ul n$, thus by induction
$$
V(\ul k) \subset \grG(m,m',m''-2) \cdot \grG(\ul n) \subset \grG(\ul m) \cdot \grG(\ul n).
$$

Notice that this covers the case $\ul k = (0,0,0)$. Indeed by Case 1 we can assume that either $\ul m$ or $\ul n$ has no zero entry. Therefore, up to a permutation of the coordinates, we always have either $m'' > |m-m'|$ or $n'' > |n-n'|$.

\textit{Case 3.} Suppose that $\ul k \neq (0,0,0)$. Up to a permutation of the coordinates we can assume that $k'' < k+k'$. 

Assume that $m'' = m + m'$, and $n'' = n+n'$. By Case 1 and by symmetry we can also assume that all the entries of $\ul m$ are nonzero.  Then $(m,m',m''-2) \in \mathrm T$. Moreover by assumption we have $k'' < k + k' \leq m+n + m'+n' = m'' + n''$. Therefore $\ul k \leq (m,m',m''-2) + \ul n$, and by induction
\[
	V(\ul k) \subset \grG(m,m',m''-2) \cdot \grG(\ul n) \subset \grG(\ul m) \cdot \grG(\ul n).
\]

Assume now that either $m'' < m + m'$ or $n'' < n+n'$. We can assume that we are in the first case. Then $(m-1,m'-1,m'')$ and $(k-1,k'-1,k'')$ both belong to $\mathrm T$ and $(k-1,k'-1,k'') < (m-1,m'-1,m'') + \ul n$, thus by induction
$$
V(k-1,k-1,k'') \subset \grG(m-1,m'-1,m'') \cdot \grG(\ul n).
$$
Therefore
$$
	V(\ul k) \subset V(1,1,0) \cdot V(k-1,k'-1,k'') \subset V(1,1,0) \cdot \grG(m-1,m'-1,m'') \cdot \grG(\ul n),
$$
and the claim follows because $V(1,1,0) \cdot \grG(m-1,m'-1,m'')  \subset \grG(\ul m)$.
\end{proof}

\begin{corollary}
The multiplication $m_{D,E}$ is surjective for all $D,E \in \mN\grD$.
\end{corollary}

\section{Normality and weight semigroups}\label{s:4}

Let $K \subset G$ be a Hermitian symmetric subgroup. As it is well known (see e.g.\ \cite[Section~5.5]{RRS}), this is equivalent to require that $K$ is the Levi factor of a parabolic subgroup $Q$ of $G$ with Abelian unipotent radical. This implies that $Q\subset G$ is a maximal parabolic subgroup (we assume $G$ to be almost simple).

Recall that we have fixed a maximal torus $T$ in $K$ and a Borel subgroup $B$ of $K$ containing $T$. We denote by $\mathcal X(T)$ the weight lattice of $T$. Recall the decomposition $\gog = \gok \oplus \gop$.

The torus $T$ is also a maximal torus in $G$. We can choose a Borel subgroup of $G$ containing $B$, the Borel subgroup of $K$, and contained in $Q$.

If $Q_-$ denotes the opposite parabolic subgroup, then we get the $K$-module decomposition
$$\gop = \gop_1 \oplus \gop_2,$$
where $\gop_1$ (resp.\ $\gop_2$) is the Lie algebra of the unipotent radical of $Q$ (resp.\ $Q_-$). Notice that $\gop_1$ and $\gop_2$ are irreducible $K$-modules, dual to each other. More precisely, if $\theta_G \in \calX(T)$ denotes the highest root of $G$ (w.r.t.\ the above choice of a Borel subgroup in $G$), then $\gop_1 = V(\theta_G)$ is the irreducible $K$-module of highest weight $\theta_G$, and $\gop_2 = V(\theta_G)^*$ is the irreducible $K$-module of lowest weight $-\theta_G$.
 
We denote by $Z_K$ the identity component of the center of $K$, and by $\goz_K$ its Lie algebra. Since $G/K$ is a Hermitian symmetric space, $\dim Z_K = 1$.

\begin{proposition}
$Z_K$ acts non-trivially on $\gop_1$ and $\gop_2$.
\end{proposition}

\begin{proof}
Let $z(\xi)$ ($\xi \in \mC^*$) be a parametrization of $Z_K \simeq \mC^*$. Since $\gop_2 \simeq \gop_1^*$, it follows that $z(\xi).e = \xi^m e_1 + \xi^{-m} e_2$ with $m \in \mZ$. Suppose that $m = 0$: then $Z_K$ acts trivially on $\gop = \gop_1 \oplus \gop_2$, hence it acts trivially on $\gog$ since it acts trivially on $\gok$. Therefore $Z_K$ is in the center of $G$, which is absurd since $G$ is semisimple.
\end{proof}

It follows that $\gop_1$ and $\gop_2$ are not isomorphic as representations of $K$. Let $\chi$ be the character of $Z_K$ acting on $\gop_1$. By making use of the classification of the standard parabolic subgroups of $G$ with Abelian unipotent radical, we can describe this character explicitly.

Indeed, let $S_G = \{\gra_1, \ldots, \gra_n\}$ be the set of simple roots of $G$, and denote by $[\theta_G : \alpha_i]$ the coefficient of $\alpha_i$ in $\theta_G$. Then a standard parabolic subgroup $Q \subset G$ has an Abelian unipotent radical if and only if it is maximal, corresponding to a root $\gra_p \in S_G$ such that $[\theta_G:\gra_p] = 1$. In the following list we give all the simple roots (of irreducible root systems of classical type) with this property:
\begin{itemize}
	\item[(1)] If $G$ is of type $\sfA_n$: $\gra_1, \ldots, \gra_n$;
	\item[(2)] If $G$ is of type $\sfB_n$: $\gra_1$;
	\item[(3)] If $G$ is of type $\sfC_n$: $\gra_n$;
	\item[(4)] If $G$ is of type $\sfD_n$: $\gra_1, \gra_{n-1}, \gra_n$.
\end{itemize}

Let $\got_G \subset \gog$ be the Cartan subalgebra generated by the fundamental coweights $\gro_1^\vee, \ldots, \gro_n^\vee$, and let $\got_K^\mss \subset \got_G$ be the subalgebra generated by the simple coroots of $K$. Since $K$ is the Levi subgroup of $G$ defined by the set of simple roots $S_G \senza \{\gra_p\}$, it follows that $\goz_K$, which is by definition the annihilator of $(\got_K^\mss)^*$ in $\got_G$, is generated by the fundamental coweight $\gro_p^\vee$.

On the other hand, assuming $G$ simply connected, we have that the cocharacter lattice $\calX(T)^\vee$ is equal to the coroot lattice $\mZ S^\vee$, therefore $$\calX(Z_K)^\vee = \goz_K \cap \mZ S^\vee
$$ 
is generated by $m \gro_p^\vee$ where $m \in \mathbb N$ is the minimum such that $m \gro_p^\vee \in \mZ S^\vee$. We list the value of $m$ here below for all the possible cases of $G$ and $\gra_p$: 
\begin{itemize}
	\item[-] $(\sfA_n, \alpha_p)$: $m = (n+1)/\gcd(p,n+1)$;
	\item[-] $(\sfB_n, \gra_1)$, $(\sfC_n, \gra_n)$, $(\sfD_n, \gra_1)$	
: $m = 2$;
	\item[-] $(\sfD_n, \gra_{n-1})$, $(\sfD_n, \gra_n)$: $m=2$ if $n$ is even, $m=4$ if $n$ is odd.
\end{itemize}
If $z(\xi)$ is the 1-parameter-subgroup in $Z_K$ given by $m\gro_p^\vee$, it follows then
$$
	\chi(z(\xi)) = \xi^{m\theta_G(\omega_p^\vee)} = \xi^m.
$$

In the following proposition we will show that if $Ke \subset \gop$ is a nilpotent orbit, then $\ol{Ke} \subset \gop$ is a bicone with respect to the decomposition $\gop = \gop_1 \oplus \gop_2$. This will allow us to study the normality of $\ol{Ke}$ by making use of Theorem~\ref{teo:normalita}.

\begin{proposition}
Write $e = e_1 + e_2$ with $e_1 \in \gop_1$ and $e_2 \in \gop_2$. Then $\xi_1 e_1 + \xi_2 e_2 \in Ke$ for all $\xi_1, \xi_2 \in \mC^*$.
\end{proposition}

\begin{proof}
Let $\{e,f,h\}$ be a normal $\mathfrak{sl}_2$-triple containing $e$, then $h \in K$ and $[h,e] = 2e$. If $t(\xi) = \exp(\xi h)$ ($\xi \in \mC^*$) is the one parameter subgroup of $K$ obtained exponentiating the line generated by $h$, it follows that $t(\xi).e = \xi e$, namely $t(\xi).e_1 = \xi e_1$ and $t(\xi).e_2 = \xi e_2$.

On the other hand, by the previous proposition, the connected component $Z_K$ of the center of $K$ acts non-trivially on $\gop_1$, therefore we can take a parametrization $z(\xi)$ of $Z$ ($\xi \in \mC^*$) such that $z(\xi).e = \xi^m e_1 + \xi^{-m} e_2$ with $m \neq 0$. It follows that every combination of $e_1$ and $e_2$ with non-zero coefficients can be written in the form $t(\xi) z(\xi').e$ for some $\xi,\xi' \in \mC^*$.
\end{proof}

Let $X$ be the wonderful compactification of $K/\mathrm N_K(K_e)$, and denote by $\Sigma$ and by $\Delta$ its set of spherical roots and its set of colors. For $i=1,2$, let $\pi_i \colon \gop \lra \gop_i$ be the projections corresponding to the decomposition $\gop = \gop_1 \oplus \gop_2$. If $\pi_i(e) \neq 0$, we denote by $D_{\gop_i} \in \mN \grD$ the element such that $\gop_i = V^*_{D_{\gop_i}}$, and the image of $\ol{Ke}$ in $\mP(\gop_i)$ coincides with the image of the corresponding map $\phi_{D_{\gop_i}} \colon X \lra \mP(\gop_i)$. In particular, if $e$ projects non-trivially both on $\gop_1$ and $\gop_2$, then the image of $\ol{Ke}$ in $\mP(\gop_1) \times \mP(\gop_2)$ coincides with the image of $X$ mapped diagonally via $\phi_{D_{\gop_1}}$ and $\phi_{D_{\gop_2}}$. For convenience we also set $D_{\gop_i} = 0$ if $\pi_i(e) = 0$, and we denote
$$
 \grD_\gop(e) = \{D_{\gop_1},  D_{\gop_2} \}.
$$

By Theorem~\ref{teo: projnorm} the multiplication of sections of globally generated line bundles on the wonderful compactification of $K/\mathrm N_K(K_e)$ is surjective, hence by Theorem~\ref{teo:normalizzazione} it follows that $\ol{Ke}\subset\gop$ is normal if and only if every $D \in \grD_\gop(e)$ is minuscule in $\mN\Delta$ with respect to the partial order $\leq_\Sigma$, or zero. Below we will see that this condition is always fulfilled, hence we get the following.

\begin{theorem}\label{teo:allnormal}
Let $(\gog, \gok)$ be a classical symmetric pair of Hermitian type and let $Ke \subset \gop$ be a spherical nilpotent orbit. Then $\ol{Ke}$ is normal.
\end{theorem}

\begin{remark}
The normality of $\overline{Ge}$ is well known and may be deduced from \cite{KP2}. 
In particular, if $(\gog, \gok)$ is a classical symmetric pair of Hermitian type and $Ke$ is a spherical nilpotent orbit in $\gop$, then $\overline{Ge}$ is always normal. 
\end{remark}

Let us denote by $\Gamma(X)$ the weight semigroup of a $K$-spherical variety $X$,
$$
	\grG(X) = \{\grl \in \mathcal X(T) \; : \; \Hom(\mC[X], V(\grl)) \neq 0 \}.
$$
Denoting for $i=1,2$ the highest weight of $\gop_i^*$ as a $G$-module by $\grl_i^*$,
the previous theorem together with Theorem~\ref{teo:normalizzazione} imply that $\grG(\ol{Ke})$ consists of the weights
$$
	 n_1\lambda_1^*+n_2\lambda_2^*-(n_1D_{\gop_1}+n_2D_{\gop_2}-E)
$$
for $(n_1,n_2)\in\mN^2$, $E \in \mN\grD$ with $E \leq_\grS n_1 D_{\gop_1} + n_2 D_{\gop_2}$, see equation \eqref{eq:weights}.

Beyond showing the normality of $\ol{Ke}$, we obtain the weight semigroups $\grG(\ol{Ke})$ by computing the corresponding semigroups
$$
\grG_{\grD_\gop(e)} = \{(n_1,n_2,E)\in\mN^2\times\mN \grD \; : \; E \leq_\grS n_1 D_{\gop_1} + n_2 D_{\gop_2} \}.
$$

The generators of the weight semigroup $\grG(\ol{Ke})=\grG(\wt{Ke})$ are given in Tables~1--6, in Appendix~\ref{B}. In the same tables we also provide the codimension of  $\overline{Ke} \smallsetminus Ke$ in $\overline{Ke}$. Notice that, if $\overline{Ke}$ is normal and the codimension of $\overline{Ke} \smallsetminus Ke$ in $\overline{Ke}$ is greater than 1, then $\mathbb C[\overline{Ke}] = \mathbb C[Ke]$, so that the weight semigroup of $Ke$ actually coincides with $\grG(\ol{Ke})$.

We now report the details of the computation of the semigroup $\grG_{\grD_\gop(e)}$. We omit the cases where $X$ is a flag variety or a parabolic induction of a wonderful symmetric variety (see Section~\ref{s:1}): in these cases the combinatorics of the spherical systems is easier. By \cite{Ki06}, the normality of $\overline{Ke}$ was already known in all these cases, since they all satisfy $\mathrm{ht}_\gop(e)=2$ (see Appendix~\ref{B}). Some of the corresponding weight semigroups $\grG(\ol{Ke})$ were obtained in \cite{Bi} by using different techniques. 

\begin{remark}
In \cite{Ni2}, for the complex symmetric pair $(\mathrm{SL}(p+q),\mathrm S(\mathrm{GL}(p)\times\mathrm{GL}(q)))$, K.~Nishiyama gave a description of the coordinate rings of the closures of some special spherical orbits, those which can be obtained as \emph{theta lift in the stable range}. Actually, in that symmetric pair, the only spherical orbits which are not theta lifts in the stable range correspond to the following cases: 1.4, 1.5, 1.6 ($r+s=q-1$) and 1.7 ($r+s=p-1$).
\end{remark}

\subsubsection*{Notation.} For all $E = \sum_{D \in \grD} k_D D \in \mZ\grD$, define its
\emph{positive part} $E^+= \sum_{k_D > 0} k_D D $
and its \emph{height} $\height(E) = \sum_{D \in
  \grD} k_D$.  

\subsection{Cases 1.4 and 1.5}

We consider the case 1.4, the other one is analogous. Let $\grS = \{\grs_1, \grs_2, \grs_3\}$ be the set of spherical roots and $\grD = \{D_1, D_2, D_3, D_4, D_5\}$ the set of colors of $X$, where we denote
$$\grs_1 = \gra', \quad \grs_2 = \gra_1, \quad \grs_3 = \gra_{p-1} $$
$$D_1 = D_{\gra'}^+, \quad D_2 = D_{\gra'}^-, \quad D_3 = D_{\gra_1}^+, \quad D_4 = D_{\gra_{p-2}}, \quad  D_5 = D_{\gra_2}$$
(if $p=4$, $D_4=D_5$).

We have in this case $\gop_1 = V(\gro_1+\gro'+\chi)$ and $\gop_2 = V(\gro_{p-1}+\gro'-\chi)$, $D_{\gop_1}=D_1$ and $D_{\gop_2}=D_2$. One can easily see that every covering difference $\grg \in \mN \grS$ is a simple root or the sum of two simple roots, thus satisfies $\height(\grg^+) = 2$. In particular every $D \in \mN\grD$ is minuscule, therefore Theorem~\ref{teo:normalita} implies that $\ol{Ke}$ is normal.

\begin{proposition}\label{prop:case1.4}
The semigroup $\grG_{\grD_\gop(e)}$ is generated by 
\[(1,0,D_1),\ (0,1,D_2),\ (1,1,D_3),\ (2,0,D_4),\ (0,2,D_5).\]
\end{proposition}

\begin{remark}	\label{oss:semigruppi}
Notice that if $D_{\gop_1}$ and $D_{\gop_2}$ are two distinct elements of $\Delta$ in order to compute generators for $\grG_{\grD_\gop(e)}$ it is actually enough to compute generators for the semigroup
\[
\grG^\grS_{\grD_\gop(e)} = \big\{\gamma\in\mN\Sigma\, :\, \supp(\gamma^+)\subset\{D_{\gop_1},D_{\gop_2}\}\big\},
\]
which is the image of the homomorphism $\grG_{\grD_\gop(e)} \lra \mN \grS$ defined by $(n_1, n_2, E) \longmapsto n_1 D_{\gop_1}+n_2 D_{\gop_2} -E$. For every generator $\grg \in \grG^\grS_{\grD_\gop(e)}$ there exists a unique minimal triple mapping to $\grg$, and $\grG_{\grD_\gop(e)}$ is generated by such triples together with $(1,0,D_{\gop_1})$ and $(0,1,D_{\gop_2})$.
\end{remark}

\begin{proof}[Proof of Proposition~\ref{prop:case1.4}]
Let us show that $\grG^\grS_{\grD_\gop(e)}$ is generated by
\[D_1+D_2-D_3=\sigma_1,\ 2D_1-D_4=\sigma_1+\sigma_3,\ 2D_2-D_5=\sigma_1+\sigma_2.\]
Indeed, these are generators of the semigroup 
\[\{a_1\sigma_1+a_2\sigma_2+a_3\sigma_3\in\mN\Sigma\, :\, a_1\geq a_2+a_3\}\]
and the condition $a_1\geq a_2+a_3$ is just equivalent to requiring the non-positivity of the coefficient of $D_3$ in
$a_1\sigma_1+a_2\sigma_2+a_3\sigma_3$ (written as an element of $\mZ\Delta$), which is equal to $-a_1+a_2+a_3$. 
\end{proof}

\subsection{Cases 1.6 and 1.7}

We consider the case 1.6, the other one is analogous. 

\subsubsection{}
We assume first $r+s < q-1$, the case $r+s = q-1$ will be treated below, separately. 

For $i =1, \ldots, r$, we denote $\grs^1_{2i-1} = \gra_{p-i}$ and $\grs^1_{2i} = \gra'_{i}$. Similarly, for $i=1, \ldots, s$, we denote $\grs^2_{2i-1} = \gra_{i}$ and $\grs^2_{2i} = \gra'_{q-i}$. Finally, we denote $\tau = \gra'_{r+1} + \ldots + \gra'_{q-s-1}$. Then
\[
	\grS = \{\grs^1_1, \ldots, \grs^1_{2r}, \grs^2_1, \ldots, \grs^2_{2s}, \tau \}.
\]

For the set of colors we introduce the following notation. For all $h \leq  2r+2$, set
$$
D^1_h = \left\{
\begin{array}{ll}
	D_{\gra_{p-i}}^- & \text{if } h = 2i-1,\text{ for }  i \leq r \\
	D_{\gra_{p-i}}^+ & \text{if } h = 2i,\text{ for } i \leq r \\
	D_{\gra'_{r}}^- & \text{if } h = 2r+1 \\
	D_{\gra_{p-r-1}} & \text{if } h = 2r+2
\end{array}\right..
$$
For all $h\leq 2s+2$, set
$$
D^2_h = \left\{
\begin{array}{ll}
	D_{\gra_{i}}^- & \text{if } h = 2i-1,\text{ for }  i  \leq s \\
	D_{\gra_{i}}^+ & \text{if } h = 2i,\text{ for }  i \leq s \\
	D_{\gra'_{q-s}}^- & \text{if } h = 2s+1 \\
	D_{\gra_{s+1}} & \text{if } h = 2s+2
\end{array}\right.
$$
Notice that if $p = r+s+1$ then $D^1_{2r+2} = D^2_{2s+2}$.

We also set
$$
D^1_{2r+3} = \left\{
\begin{array}{ll}
	D_{\gra'_{s+1}}  & \text{if } r+s < q-2\\
	D_{\tau}^+  & \text{if } r+s = q-2
\end{array}\right.,
\quad 
D^2_{2s+3} = \left\{
\begin{array}{ll}
	D_{\gra'_{q-s-1}}  & \text{if } r+s < q-2	\\
	D_{\tau}^-  & \text{if } r+s = q-2
\end{array}\right.
$$
(if $r+s = q-2$, the spherical root $\tau$ is equal to a simple root, $\tau=\alpha'_{r+1}=\alpha'_{q-s-1}$, we assume $c(D_{\tau}^+,\gra'_{r})=c(D_{\tau}^-,\gra'_{q-s})=-1$).

Therefore
$$
	\grD = \{D^1_1, \ldots, D^1_{2r+3}, \ D^2_1, \ldots, D^2_{2s+3} \}.
$$

Let us suppose $r,s>0$. We have $D_{\gop_1} = D^1_2$ and $D_{\gop_2} = D^2_2$. As explained at the end of Section~\ref{ss:General reductions}, $X$ is a parabolic induction of a quotient of a localization of a comodel wonderful variety of cotype $\sfA$, therefore by \cite[Proposition~3.2]{BGM} every covering difference $\grg \in \mN \grS$ satisfies $\height(\grg^+) = 2$. In particular every element $D \in \grD$ is minuscule, therefore Theorem~\ref{teo:normalita} implies that $\ol{Ke}$ is normal.

For notational purposes, set $r_1 = r$ and $r_2 = s$. For $k=1,2$ and $h \leq 2r_k+2$, we denote
$$
\tilde D^k_h = \left\{
\begin{array}{ll}
	D^k_h & \text{if } h < 2r_k+1 \\
	D^k_h + D^k_{h+1} & \text{if } h = 2r_k+1, 2r_k+2
\end{array}\right.
$$

\begin{proposition}
The semigroup $\grG_{\grD_\gop(e)}$ is generated by the elements $(i,0,\tilde D^1_{2i})$ for $i \leq r_1+1$, $(0,j,\tilde D^2_{2j})$ for $j \leq r_2+1$ and $(i,j,\tilde D^1_{2i-1} + \tilde D^2_{2j-1})$ for $i \leq r_1+1$, $j \leq r_2+1$.
\end{proposition}

\begin{proof}
As noticed in Remark \ref{oss:semigruppi}, it is enough to compute generators for the semigroup 
\[
\grG^\grS_{\grD_\gop(e)} = \big\{\gamma\in\mN\Sigma\, :\, \supp(\gamma^+)\subset\{D^1_2,D^2_2\}\big\}.
\]

Notice that, for $k=1,2$ and $i = 2, \ldots, r_k+1$, it holds
$$\grs^k_1+\ldots+\grs^k_{2i-2}= D^k_2 + \tilde D^k_{2i-2} - \tilde D^k_{2i}.$$
Therefore,
\[
\gamma^k_i := \sum_{u=1}^{i-1}(i-u)(\sigma^k_{2u-1}+\sigma^k_{2u})
\]
is equal to $i D^k_2-\tilde D^k_{2i}$.

Notice also that, for $i \leq r_1+1$ and $j \leq r_2+1$, 
$$
\sum_{u=i}^{r_1} \grs^1_{2u} + \sum_{v=j}^{r_2} \grs^2_{2v} + \tau =  
\tilde D^1_{2i} + \tilde D^2_{2j} - \tilde D^1_{2i-1} - \tilde D^2_{2j-1}.
$$
Therefore,
\[
\gamma_{i,j} :=
\sum_{u=1}^{i-1}(i-u)(\sigma^1_{2u-1}+\sigma^1_{2u})+\sum_{u=i}^{r_1}\sigma^1_{2u}+
\sum_{v=1}^{j-1}(j-v)(\sigma^2_{2v-1}+\sigma^2_{2v})+\sum_{v=j}^{r_2}\sigma^2_{2v}+\tau
\]
is equal  to $i D^1_{2} + j D^2_{2} - \tilde D^1_{2i-1} - \tilde D^2_{2j-1}$.

We claim that the semigroup $\grG^\grS_{\grD_\gop(e)}$ is generated by the elements of the form $\gamma^k_i$, for $k=1,2$ and $2\leq i\leq r_k+1$, and $\gamma_{i,j}$, for $1 \leq i \leq r_1+1$ and $1\leq j \leq r_2+1$.

Let us write $\grg = \sum_{h=1}^{2r_1} a^1_h \grs^1_h + \sum_{h=1}^{2r_2} a^2_h \grs^2_h + b \tau$ as an element of $\mN\Sigma$. Let us denote by $d^k_h$ the coefficient of $D^k_h$ in $\gamma$ (written as element of $\mZ\Delta$), for $h\neq 2r_k+2$ there is no ambiguity. We have
\[d^k_1=a^k_1-a^k_2,\quad d^k_h=-a^k_{h-2}+a^k_{h-1}+a^k_h-a^k_{h+1}\ (3\leq h\leq 2r_k-1),\]
\[d^k_{2r_k}=-a^k_{2r_k-2}+a^k_{2r_k-1}+a^k_{2r_k},\quad d^k_{2r_k+1}=-a^k_{2r_k-1}+a^k_{2r_k}-b,\quad d^k_{2r_k+3}=-a^k_{2r_k}+b.\]

Furthermore, every spherical root lies in the lattice generated by $\tilde D^k_1, \ldots, \tilde D^k_{2r_k+2}$, with $k\in \{1,2\}$. 
Denoting by $\tilde d^k_h$ the coefficient of $\tilde D^k_h$ in $\gamma$, we have $\tilde d^k_{2r_k+1}=d^k_{2r_k+1}$ and $\tilde d^k_{2r_k+2}=d^k_{2r_k+3}$, with $k\in \{1,2\}$.

Assume $\supp(\gamma^+)\subset\{D^1_2,D^2_2\}$, then the coefficients $\tilde d^k_h$ are non-positive for $h\neq2$. Let us write $\gamma$ as a combination with non-negative integer coefficients of the $\gamma^k_i$ ($k=1,2$ and $2\leq i\leq r_k+1$) and the $\gamma_{i,j}$ ($1\leq i \leq r_1+1$ and $1\leq j \leq r_2+1$).

We have 
$$
	\sum_{i=1}^{r_1+1} \tilde d^1_{2i-1} = - b = \sum_{j=1}^{r_2+1} \tilde d^2_{2j-1}.
$$

Therefore, there exist integers $c_{i,j}\leq0$ (for $i \leq r_1+1$ and $j \leq r_2+1$) such that $\sum_j c_{i,j}=\tilde d^1_{2i-1}$ and $\sum_i c_{i,j}=\tilde d^2_{2j-1}$. Indeed, for $k=1,2$ and $i \leq r_k$, we can set $n_i^k = -\sum_{u=1}^i \tilde d^k_{2u-1}$ and take 
$$
	- c_{i,j} = \mathrm{card} \{ n \in \mN \; | \; n_{i-1}^1 < n \leq n_i^1 \text{ and } n_{j-1}^2 < n \leq n_j^2 \}.
$$

We claim that $\gamma$ is equal to
\[
\sum_{i=2}^{r_1+1}-\tilde d^1_{2i}\gamma^1_i+\sum_{j=2}^{r_2+1}-\tilde d^2_{2j}\gamma^2_j+\sum_{i=1}^{r_1+1}\sum_{j=1}^{r_2+1}-c_{i,j}\gamma_{i,j}.
\]

Indeed, the coefficient of $\sigma^1_{2i-1}$ in the above expression is equal to
\begin{eqnarray*}
&&\sum_{u=i+1}^{r_1+1}-\tilde d^1_{2u}(u-i)+\sum_{u=i+1}^{r_1+1}\sum_{v=1}^{r_2+1}-c_{u,v}(u-i)\\
&=&\sum_{u=i+1}^{r_1+1}-(\tilde d^1_{2u}+\tilde d^1_{2u-1})(u-i)\\
&=&a^1_{2i-1}.
\end{eqnarray*}
The coefficient of $\sigma^1_{2i}$ is equal to
\begin{eqnarray*}
&&\sum_{u=i+1}^{r_1+1}-\tilde d^1_{2u}(u-i)+\sum_{u=1}^i\sum_{v=1}^{r_2+1}-c_{u,v}+\sum_{u=i+1}^{r_1+1}\sum_{v=1}^{r_2+1}-c_{u,v}(u-i)\\
&=&\sum_{u=1}^i-\tilde d^1_{2u-1}+\sum_{u=i+1}^{r_1+1}-(\tilde d^1_{2u}+\tilde d^1_{2u-1})(u-i)\\
&=&a^1_{2i}.
\end{eqnarray*}
Analogously, the same holds for $\sigma^2_h$, for any $h$. It remains the coefficient of $\tau$, which is equal to
\[\sum_{i=1}^{r_1+1}\sum_{j=1}^{r_2+1}-c_{i,j}=b.	\qedhere\]
\end{proof}

The case $r=s=0$ is a parabolic induction of a wonderful symmetric variety. 
We are left with the case $r>0$ and $s=0$ (the other one, $r=0$ and $s>0$, is analogous).
Let us keep the same notation as above, notice that there exists no $D^2_1$ and we have
$$
\grD = \{D^1_1, \ldots, D^1_{2r+3}\}\cup\{ D^2_2,  D^2_3 \}.
$$
In this case, $D_{\gop_1} = D^1_2$ and $D_{\gop_2} = \tilde D^2_2 = D^2_2+D_3^2$. Both are minuscule, and $\ol{Ke}$ is normal.

The description of the $\grG_{\grD_\gop(e)}$ given in the above proposition remains valid. The proof is slightly simpler: every spherical root lies in the lattice generated by $\tilde D^1_1, \ldots, \tilde D^1_{2r+2}$ and $\tilde D^2_1, \tilde D^2_2$ which are still linearly independent, denoting by $\tilde d^k_h$ the coefficient of $\tilde D^k_h$ in $\gamma$, the semigroup
\[\big\{\gamma\in\mN\Sigma\, :\, \tilde d^1_h\leq0\ \forall\ h\neq2\big\}\] 
is generated by the elements of the form $\gamma^1_i$, for $2\leq i\leq r+1$, and $\gamma_{i,1}$, for $1 \leq i \leq r+1$. 

\subsubsection{}
We now consider the case $r+s=q-1$.

Let us keep the same notation as above, as far as possible. Indeed, there exists no $\tau$, so we have 
\[
	\grS = \{\grs^1_1, \ldots, \grs^1_{2r}, \grs^2_1, \ldots, \grs^2_{2s} \}
\]
and
\[
	\grD = \{D^1_1, \ldots, D^1_{2r+2}, \ D^2_1, \ldots, D^2_{2s+2} \}.
\]

Let us suppose $r,s>0$. We have $D_{\gop_1} = D^1_2$ and $D_{\gop_2} = D^2_2$, which as in previous case are minuscule. Therefore, $\ol{Ke}$ is normal.

For convenience we also define $D^1_{2r+3} = D^2_{2s+1}$ and $D^2_{2s+3} = D^1_{2r+1}$. 
As in previous case, set $r_1 = r$ and $r_2 = s$. If $k=1,2$ and $h \leq 2r_k+2$, denote
$$
\tilde D^k_h = \left\{
\begin{array}{ll}
	D^k_h & \text{if } h < 2r_k+1 \\
	D^k_h + D^k_{h+1} & \text{if } h = 2r_k+1, 2r_k+2
\end{array}\right.
$$
Notice that if $p=q+1$ then $D^1_{2r_1+2}=D^2_{2r_2+2}$, thus $\tilde D^1_{2r_1+2}=\tilde D^2_{2r_2+1}$ and $\tilde D^2_{2r_2+2}=\tilde D^1_{2r_1+1}$.

\begin{proposition}
The semigroup $\grG_{\grD_\gop(e)}$ is generated by the elements  $(i,0,\tilde D^1_{2i})$ for $i \leq r_1+1$, $(0,j,\tilde D^2_{2j})$ for $j \leq r_2+1$ and $(i,j,\tilde D^1_{2i-1} + \tilde D^2_{2j-1})$ for $i \leq r_1+1$, $j \leq r_2+1$ with $i+j < r_1+r_2+2$.
\end{proposition}

\begin{proof}
We follow the line of the proof of the previous proposition. As in that case, it is enough to compute generators for $\grG^\grS_{\grD_\gop(e)}$.

For $k=1,2$ and $i = 2, \ldots, r_k+1$, we have
\[
\gamma^k_i := \sum_{u=1}^{i-1}(i-u)(\sigma^k_{2u-1}+\sigma^k_{2u}) = i D^k_2-\tilde D^k_{2i}.
\]

For $i \leq r_1+1$ and $j \leq r_2+1$ with $i+j<r_1+r_2+2$, we have
\begin{eqnarray*}
\gamma_{i,j} & := &
\sum_{u=1}^{i-1}(i-u)(\sigma^1_{2u-1}+\sigma^1_{2u})+\sum_{u=i}^{r_1}\sigma^1_{2u}+
\sum_{v=1}^{j-1}(j-v)(\sigma^2_{2v-1}+\sigma^2_{2v})+\sum_{v=j}^{r_2}\sigma^2_{2v} \\
& = & i D^1_{2} + j D^2_{2} - \tilde D^1_{2i-1} - \tilde D^2_{2j-1}.
\end{eqnarray*}

Let us prove that $\grG^\grS_{\grD_\gop(e)}$ is generated by the elements of the form $\gamma^k_i$, for $k=1,2$ and $2\leq i\leq r_k+1$, and $\gamma_{i,j}$, for $1 \leq i \leq r_1+1$ and $1\leq j \leq r_2+1$ with $i+j<r_1+r_2+2$. 

Let us write $\grg = \sum_{h=1}^{2r_1} a^1_h \grs^1_h + \sum_{h=1}^{2r_2} a^2_h \grs^2_h\in\mN\Sigma$ and denote by $d^k_h$ the coefficient of $D^k_h$ in $\gamma$, for $h< 2r_k+2$. We have
\[d^k_1=a^k_1-a^k_2,\quad d^k_h=-a^k_{h-2}+a^k_{h-1}+a^k_h-a^k_{h+1}\ (3\leq h\leq 2r_k-1),\]
\[d^k_{2r_k}=-a^k_{2r_k-2}+a^k_{2r_k-1}+a^k_{2r_k},\]
\[d^1_{2r_1+1}=-a^1_{2r_1-1}+a^1_{2r_1}-a^2_{2r_2},\quad d^2_{2r_2+1}=-a^2_{2r_2-1}+a^2_{2r_2}-a^1_{2r_1}.\]

Assume $\supp(\gamma^+)\subset\{D^1_2,D^2_2\}$, then the coefficients $d^k_h$ are non-positive for $h\neq2$.

We have 
$$
\sum_{i=1}^{r_k} d^k_{2i-1} = a^k_{2r_k-1} - a^k_{2r_k},\quad
\sum_{i=1}^{r_1+1} d^1_{2i-1} = - a^2_{2r_2},\quad
\sum_{j=1}^{r_2+1} d^2_{2j-1} = - a^1_{2r_1},
$$
and moreover
$$
d^1_{2r_1+1}+d^2_{2r_2+1}=-a^1_{2r_1-1}-a^2_{2r_2-1}.
$$
Therefore, there exist non-positive integers $c^1$, $c^2$ and $c_{i,j}$, for $i \leq r_1+1$ and $j \leq r_2+1$ with $i+j<r_1+r_2+2$, such that 
\[\sum_{j=1}^{r_2+1} c_{i,j}=d^1_{2i-1}\ \forall\ i\leq r_1,\qquad \sum_{i=1}^{r_1+1} c_{i,j}=d^2_{2j-1}\ \forall\ j\leq r_2,\]
\[c^1+\sum_{i=1}^{r_1} c_{i,r_2+1}=d^2_{2r_2+1},\qquad c^2+\sum_{j=1}^{r_2} c_{r_1+1,j}=d^1_{2r_1+1},\]
\[c^1+\sum_{j=1}^{r_2} c_{r_1+1,j}= -a^1_{2r_1-1},\quad\text{and}\quad c^2+\sum_{i=1}^{r_1} c_{i,r_2+1}= -a^2_{2r_2-1}.\]
Indeed, if we assume (without loss of generality) $a^1_{2r_1}\leq a^2_{2r_2}$, we can take 
$c^1=a^1_{2r_1}-a^2_{2r_2}-b$ and $c^2=-b$ where $b=\min(a^1_{2r_1-1},a^2_{2r_2-1}-a^2_{2r_2}+a^1_{2r_1})$. For the $c_{i,j}$'s one can do as in the proof of the previous proposition.

We claim that $\gamma$ is equal to
\[
\big(\sum_{i=2}^{r_1} -d^1_{2i} \gamma^1_i\big) -c^1 \gamma^1_{r_1+1} + \big(\sum_{j=2}^{r_2} -d^2_{2j} \gamma^2_j\big) -c^2 \gamma^2_{r_2+1} + \sum_{\substack{1\leq i\leq r_1+1\\ 1\leq j\leq r_2+1\\ i+j<r_1+r_2+2}}-c_{i,j}\gamma_{i,j}.
\]

Indeed, the coefficient of $\sigma^1_{2i-1}$ in the above expression is equal to
\begin{eqnarray*}
&&\big(\sum_{u=i+1}^{r_1} -d^1_{2u}(u-i)\big) -c^1(r_1+1-i) + \sum_{\substack{i+1\leq u\leq r_1+1\\1\leq v\leq r_2+1\\ u+v<r_1+r_2+2}}-c_{u,v}(u-i)\\
&&=\big(\sum_{u=i+1}^{r_1}-(d^1_{2u}+ d^1_{2u-1})(u-i)\big) + a^1_{2r_1-1}(r_1+1-i)\\
&&=a^1_{2i-1}.
\end{eqnarray*}
The coefficient of $\sigma^1_{2i}$ is equal to
\begin{eqnarray*}
&&\big(\sum_{u=i+1}^{r_1} -d^1_{2u}(u-i)\big) -c^1(r_1+1-i) + \big(\sum_{u=1}^i\sum_{v=1}^{r_2+1} -c_{u,v}\big) + \!\!\!\!\!\!\sum_{\substack{i+1\leq u\leq r_1+1\\1\leq v\leq r_2+1\\ u+v<r_1+r_2+2}}\!\!\!\!\!\!-c_{u,v}(u-i)\\
&&=\big(\sum_{u=1}^i -d^1_{2u-1}\big) + \big(\sum_{u=i+1}^{r_1}-(d^1_{2u}+d^1_{2u-1})(u-i)\big) + a^1_{2r_1-1}(r_1+1-i)\\
&&=a^1_{2i}.
\end{eqnarray*}
The same holds for $\sigma^2_h$, for any $h$.
\end{proof}

The case $r=s=0$ corresponds to a flag variety. 
We are left with the case $r>0$ and $s=0$ (the other one, $r=0$ and $s>0$, is analogous).
Keeping the same notation as above, there exists no $D^2_1$ and we have
$$
\grD = \{D^1_1, \ldots, D^1_{2r+2}\}\cup\{ D^2_2 \}.
$$
In this case, $D_{\gop_1} = D^1_2$ and $D_{\gop_2} = \tilde D^2_2 = D^2_2+D^1_{2r+1}$. Both are minuscule, and $\ol{Ke}$ is normal.

The description of the semigroup $\grG_{\grD_\gop(e)}$ remains the same: denoting by $d^k_h$ the coefficient of $D^k_h$ in $\gamma$, for $h\leq2r$, the semigroup
\[\big\{\gamma\in\mN\Sigma\, :\, d^1_h\leq0\ \forall\ h\neq2\big\}\] is generated by the elements of the form $\gamma^1_i$, for $2\leq i\leq r+1$, and $\gamma_{i,1}$, for $1 \leq i \leq r$.


\appendix

\section{List of spherical nilpotent $K$-orbits  in $\mathfrak p$\\in the classical Hermitian cases}\label{A}

\renewcommand{\thesubsection}{\arabic{subsection}}

Here we report the list of the spherical nilpotent
$K$-orbits in $\mathfrak p$ for all symmetric pairs $(\mathfrak
g,\mathfrak k)$ of classical Hermitian type.

Every (complex) $K$-orbit in $\mathfrak p$ is labelled with the signed partition of the corresponding real nilpotent orbit.
In each case we provide a normal triple $\{h,e,f\}$, with $e$ a representative of the orbit.

We denote by $Q$ the parabolic subgroup of $K$ whose Lie algebra is equal to
\[\mathrm{Lie}\,Q=\bigoplus_{i\geq0}\mathfrak k(i),\]
where $\mathfrak k(i)$ is the $\mathrm{ad}h$-eigenspace in $\mathfrak k$ of eigenvalue $i$. 

We describe the centralizer of $h$, denoted by $K_h$ or by $L$, which is a Levi subgroup of $Q$. Let $Q^\mathrm u$ be the unipotent radical of $Q$. Then we describe the centralizer of $e$, denoted by $K_e$. A Levi subgroup of $K_e$ is always given by $L_e$, the centralizer of $e$ in $L$. The unipotent radical of $K_e$ is explicitly described as $L_e$-submodule of $Q^\mathrm u$.

\subsection{$\mathrm{SL}(p+q)/\mathrm S(\mathrm{GL}(p)\times\mathrm{GL}(q))$}\

$K=\mathrm S(\mathrm{GL}(p)\times\mathrm{GL}(q))$, $p,q\geq2$, $\mathfrak p=V(\omega_1+\omega'_{q-1})\oplus V(\omega_{p-1}+\omega'_1)$ as $K^{\mathrm{ss}}$-module.
If $p=1$ and $q\geq2$, $\mathfrak p=V(\omega'_{q-1})\oplus V(\omega'_1)$.
If $p\geq2$ and $q=1$, $\mathfrak p=V(\omega_1)\oplus V(\omega_{p-1})$.
If $p=q=1$, $\mathfrak p=V(0)\oplus V(0)$.

Let us fix a basis $e_1,\ldots,e_p$ of $\mathbb C^{p}$ and denote by $\varphi_1,\ldots,\varphi_p$ the dual basis of $(\mathbb C^{p})^*$. Similarly, let us fix a basis $e'_1,\ldots,e'_q$ of $\mathbb C^{q}$ and denote by $\varphi'_1,\ldots,\varphi'_q$ the dual basis of $(\mathbb C^{q})^*$. Then $K=\mathrm S(\mathrm{GL}(\mathbb C^{p})\times\mathrm{GL}(\mathbb C^{q}))$ and 
\[\mathfrak p=\big(\mathbb C^{p}\otimes(\mathbb C^{q})^*\big)\oplus\big((\mathbb C^{p})^*\otimes\mathbb C^{q}\big).\]

\subsubsection*{1.1. $\mathbf{(+2^r,+1^{p-r},-1^{q-r})}$, $r\geq1$}\

\[e=\sum_{i=1}^re_i\otimes \varphi'_{q-r+i},\qquad
f=\sum_{i=1}^r\varphi_{i}\otimes e'_{q-r+i},\] 
\[h(e_i)=\left\{\begin{array}{cl}
e_i & \mbox{if $1\leq i\leq r$}\\
0 & \mbox{otherwise}
\end{array}\right.,\quad
h(e'_i)=\left\{\begin{array}{cl}
-e'_i & \mbox{if $q-r+1\leq i\leq q$}\\
0 & \mbox{otherwise}
\end{array}\right..\]
Let $Q=L\,Q^\mathrm u$ be the corresponding parabolic subgroup of $K$,
so that $L=K_h\cong\mathrm{S}(\mathrm{GL}(r)\times\mathrm{GL}(p-r)\times\mathrm{GL}(q-r)\times\mathrm{GL}(r))$.

The centralizer of $e$ is $K_e=L_eQ^\mathrm u$ where
$L_e\cong\mathrm{S}(\mathrm{GL}(r)\times\mathrm{GL}(p-r)\times\mathrm{GL}(q-r))$,
the $\mathrm{GL}(r)$ factor of $L_e$ is embedded diagonally, $A\mapsto(A,A)$, into the $\mathrm{GL}(r)\times\mathrm{GL}(r)$ factor of $L$.
For $r=p=q$, the connected component of $L_e$ is isomorphic to $\mathrm{SL}(r)$.

\subsubsection*{1.2. $\mathbf{(-2^r,+1^{p-r},-1^{q-r})}$, $r\geq1$}\

\[e=\sum_{i=1}^r\varphi_{p-r+i}\otimes e'_i,\qquad
f=\sum_{i=1}^re_{p-r+i}\otimes \varphi'_i,\] 
\[h(e_i)=\left\{\begin{array}{cl}
-e_i & \mbox{if $p-r+1\leq i\leq p$}\\
0 & \mbox{otherwise}
\end{array}\right.,\quad
h(e'_i)=\left\{\begin{array}{cl}
e'_i & \mbox{if $1\leq i\leq r$}\\
0 & \mbox{otherwise}
\end{array}\right..\]
The centralizers of $h$ and $e$ are the same as in the previous case up to switching the two factors of $K$ as well as the role of $p$ and $q$.



\subsubsection*{1.3. $\mathbf{(+2^r,-2^s,+1^{p-r-s},-1^{q-r-s})}$, $r,s\geq1$}\

\[e=\sum_{i=1}^re_i\otimes \varphi'_{q-r+i}+\sum_{i=1}^s\varphi_{p-s+i}\otimes e'_i,\]
\[h(e_i)=\left\{\begin{array}{cl}
e_i & \mbox{if $1\leq i\leq r$}\\
-e_i & \mbox{if $p-s+1\leq i\leq p$}\\
0 & \mbox{otherwise}
\end{array}\right.,\quad
h(e'_i)=\left\{\begin{array}{cl}
e'_i & \mbox{if $1\leq i\leq s$}\\
-e'_i & \mbox{if $q-r+1\leq i\leq q$}\\
0 & \mbox{otherwise}
\end{array}\right.,\]
\[f=\sum_{i=1}^r\varphi_{i}\otimes e'_{q-r+i}+\sum_{i=1}^se_{p-s+i}\otimes \varphi'_i.\] 
Let $Q=L\,Q^\mathrm u$ be the corresponding parabolic subgroup of $K$,
so that $L=K_h\cong\mathrm{S}(\mathrm{GL}(r)\times\mathrm{GL}(p-r-s)\times\mathrm{GL}(s)\times\mathrm{GL}(s)\times\mathrm{GL}(q-r-s)\times\mathrm{GL}(r))$.

The centralizer of $e$ is $K_e=L_eQ^\mathrm u$ where
$L_e\cong\mathrm{S}(\mathrm{GL}(r)\times\mathrm{GL}(p-r-s)\times\mathrm{GL}(s)\times\mathrm{GL}(q-r-s))$,
the $\mathrm{GL}(r)$ and $\mathrm{GL}(s)$ factors of $L_e$ are embedded diagonally, respectively, into the $\mathrm{GL}(r)\times\mathrm{GL}(r)$ and $\mathrm{GL}(s)\times\mathrm{GL}(s)$ factors of $L$.

\subsubsection*{1.4. $\mathbf{(+3^2,+1^{p-4})}$, $q=2$}\

\[e=e_1\otimes \varphi'_1+e_2\otimes \varphi'_2+\varphi_{p-1}\otimes e'_1+\varphi_p\otimes e'_2,\]
\[h(e_i)=\left\{\begin{array}{cl}
2e_i & \mbox{if $1\leq i\leq 2$}\\
-2e_i & \mbox{if $p-1\leq i\leq p$}\\
0 & \mbox{otherwise}
\end{array}\right.,\quad
h(e'_i)=0\ \forall\ i,\]
\[f=2(\varphi_1\otimes e'_1+\varphi_2\otimes e'_2+e_{p-1}\otimes \varphi'_1+e_p\otimes \varphi'_2).\] 
Let $Q=L\,Q^\mathrm u$ be the corresponding parabolic subgroup of $K$,
so that $L=K_h\cong\mathrm{S}(\mathrm{GL}(2)\times\mathrm{GL}(p-4)\times\mathrm{GL}(2)\times\mathrm{GL}(2))$.

The centralizer of $e$ is $K_e=L_eQ^\mathrm u$ where
$L_e\cong\mathrm{S}(\mathrm{GL}(2)\times\mathrm{GL}(p-4))$,
the $\mathrm{GL}(2)$ factor of $L_e$ is embedded diagonally, $A\mapsto(A,A,A)$, into the $\mathrm{GL}(2)\times\mathrm{GL}(2)\times\mathrm{GL}(2)$ factor of $L$.
For $p=4$, the connected component of $L_e$ is isomorphic to $\mathrm{SL}(2)$.

\subsubsection*{1.5. $\mathbf{(-3^2,-1^{q-4})}$, $p=2$}\

\[e=e_1\otimes \varphi'_{q-1}+e_2\otimes \varphi'_q+\varphi_1\otimes e'_1+\varphi_2\otimes e'_2,\]
\[h(e_i)=0\ \forall\ i,\quad
h(e'_i)=\left\{\begin{array}{cl}
2e'_i & \mbox{if $1\leq i\leq 2$}\\
-2e'_i & \mbox{if $q-1\leq i\leq q$}\\
0 & \mbox{otherwise}
\end{array}\right.,\]
\[f=2(\varphi_1\otimes e'_{q-1}+\varphi_2\otimes e'_q+e_1\otimes \varphi'_1+e_2\otimes \varphi'_2).\] 
The centralizers of $h$ and $e$ are the same as in the previous case up to switching the two factors of $K$ as well as the role of $p$ and $q$.



\subsubsection*{1.6. $\mathbf{(+3,+2^r,-2^s,+1^{p-r-s-2},-1^{q-r-s-1})}$}\

\[e=e_1\otimes\varphi'_{q-r}+\sum_{i=1}^re_{i+1}\otimes \varphi'_{q-r+i}+\sum_{i=1}^s\varphi_{p-s+i-1}\otimes e'_i+\varphi_p\otimes e'_{q-r},\]
\[h(e_i)=\left\{\begin{array}{cl}
2e_i & \mbox{if $i=1$}\\
e_i & \mbox{if $2\leq i\leq r+1$}\\
-e_i & \mbox{if $p-s\leq i\leq p-1$}\\
-2e_i & \mbox{if $i=p$}\\
0 & \mbox{otherwise}
\end{array}\right.,\quad
h(e'_i)=\left\{\begin{array}{cl}
e'_i & \mbox{if $1\leq i\leq s$}\\
-e'_i & \mbox{if $q-r+1\leq i\leq q$}\\
0 & \mbox{otherwise}
\end{array}\right.,\]
\[f=2\varphi_1\otimes e'_{q-r}+\sum_{i=1}^r\varphi_{i+1}\otimes e'_{q-r+i}+\sum_{i=1}^se_{p-s+i-1}\otimes \varphi'_i+2e_p\otimes\varphi'_{q-r}.\] 
Let $Q=L\,Q^\mathrm u$ be the corresponding parabolic subgroup of $K$,
so that $L=K_h\cong\mathrm{S}(\mathrm{GL}(1)\times\mathrm{GL}(r)\times\mathrm{GL}(p-r-s-2)\times\mathrm{GL}(s)\times\mathrm{GL}(1)\times\mathrm{GL}(s)\times\mathrm{GL}(q-r-s)\times\mathrm{GL}(r))$.

The centralizer of $e$ is $K_e=L_eK_e^\mathrm u$ where
$L_e\cong\mathrm{S}(\mathrm{GL}(1)\times\mathrm{GL}(r)\times\mathrm{GL}(p-r-s-2)\times\mathrm{GL}(s)\times\mathrm{GL}(q-r-s-1))$,
the $\mathrm{GL}(1)\times\mathrm{GL}(q-r-s-1)$ factor of $L_e$ is embedded as
\[(z,A)\mapsto(z,z,(A,z))\]
into $\mathrm{GL}(1)\times\mathrm{GL}(1)\times(\mathrm{GL}(q-r-s-1)\times\mathrm{GL}(1))$ 
and $\mathrm{GL}(q-r-s-1)\times\mathrm{GL}(1)$ is included in the $\mathrm{GL}(q-r-s)$ factor of $L$,
the $\mathrm{GL}(r)$ and $\mathrm{GL}(s)$ factors of $L_e$ are embedded diagonally, respectively, into the $\mathrm{GL}(r)\times\mathrm{GL}(r)$ and $\mathrm{GL}(s)\times\mathrm{GL}(s)$ factors of $L$.
The quotient $\mathrm{Lie}\,Q^\mathrm u/\mathrm{Lie}\,K_e^\mathrm u$ is the sum of two simple $L_e$-modules of dimension $r$ and $s$, respectively, as follows. In $\mathfrak k(1)$ there are exactly two simple $L_e$-submodules, $W_{0,1},W_{1,1}$, of highest weight $\omega_{r-1}$ w.r.t.\ the semisimple part of the $\mathrm{GL}(r)$ factor, isomorphic as $L_e$-modules but lying in two distinct isotypical $L$-components. Similarly, in $\mathfrak k(1)$ there are exactly two simple $L_e$-submodules, $W_{0,2},W_{1,2}$, of highest weight $\omega_{1}$ w.r.t.\ the semisimple part of the $\mathrm{GL}(s)$ factor, isomorphic as $L_e$-modules but lying in two distinct isotypical $L$-components. Let $V$ be the $L_e$-complement of $W_{0,1}\oplus W_{1,1}\oplus W_{0,2}\oplus W_{1,2}$ in $\mathrm{Lie}\,Q^\mathrm u$. As $L_e$-module, $\mathrm{Lie}\,K_e^\mathrm u$ is the direct sum of $V$, of a simple $L_e$-submodule of $W_{0,1}\oplus W_{1,1}$ which projects non-trivially on both summands $W_{0,1}$ and $W_{1,1}$, and of a simple $L_e$-submodule of $W_{0,2}\oplus W_{1,2}$ which projects non-trivially on both summands $W_{0,2}$ and $W_{1,2}$.

\subsubsection*{1.7. $\mathbf{(-3,+2^r,-2^s,+1^{p-r-s-1},-1^{q-r-s-2})}$}\

\[e=\sum_{i=1}^re_i\otimes \varphi'_{q-r+i-1}+e_{p-s}\otimes\varphi'_q+\varphi_{p-s}\otimes e'_1+\sum_{i=1}^s\varphi_{p-s+i}\otimes e'_{i+1},\]
\[h(e_i)=\left\{\begin{array}{cl}
e_i & \mbox{if $1\leq i\leq r$}\\
-e_i & \mbox{if $p-s+1\leq i\leq p$}\\
0 & \mbox{otherwise}
\end{array}\right.,\quad
h(e'_i)=\left\{\begin{array}{cl}
2e'_i & \mbox{if $i=1$}\\
e'_i & \mbox{if $2\leq i\leq s+1$}\\
-e'_i & \mbox{if $q-r\leq i\leq q-1$}\\
-2e'_i & \mbox{if $i=q$}\\
0 & \mbox{otherwise}
\end{array}\right.,\]
\[f=\sum_{i=1}^r\varphi_i\otimes e'_{q-r+i-1}+2\varphi_{p-s}\otimes e'_q+2e_{p-s}\otimes\varphi'_1+\sum_{i=1}^se_{p-s+i}\otimes \varphi'_{i+1}.\] 
The centralizers of $h$ and $e$ are the same as in the previous case up to switching the two factors of $K$ as well as the role of $p$ and $q$, and the role of $r$ and $s$, respectively.

\subsection{$\mathrm{SO}(2n+1)/\mathrm{SO}(2n-1)\times\mathrm{SO}(2)$}\

$K=\mathrm{SO}(2n-1)\times\mathrm{SO}(2)$, $n>2$, $\mathfrak p=V(\omega_1)\oplus V(\omega_1)$ as $K^{\mathrm{ss}}$-module. 

Let us fix a basis $e_1,\ldots,e_{n-1},e_0,e_{-n+1},\ldots,e_{-1}$ of $\mathbb C^{2n-1}$, a symmetric bilinear form $\beta$ such that $\beta(e_i,e_j)=\delta_{i,-j}$ for all $i,j$. Similarly, let us fix a basis $e'_1,e'_{-1}$ of $\mathbb C^{2}$ and a symmetric bilinear form $\beta'$ such that $\beta'(e'_i,e'_j)=\delta_{i,-j}$ for all $i,j$. For convenience, let us denote by $\varphi'_1,\varphi'_{-1}$ the dual basis of $(\mathbb C^{2})^*$. Then $K=\mathrm{SO}(\mathbb C^{2n-1},\beta)\times\mathrm{SO}(\mathbb C^{2},\beta')$ and 
\[\mathfrak p=\mathbb C^{2n-1}\otimes(\mathbb C^{2})^*.\] 

\subsubsection*{2.1. $\mathbf{(+2^2,+1^{2n-3})}$, $I$ and $II$}\

Case (I)
\[e=e_1\otimes \varphi'_{-1},\qquad
f=-e_{-1}\otimes \varphi'_{1},\] 
\[h(e_i)=\left\{\begin{array}{cl}
e_i & \mbox{if $i=1$}\\
-e_i & \mbox{if $i=-1$}\\
0 & \mbox{otherwise}
\end{array}\right.,\quad
h(e'_i)=\left\{\begin{array}{cl}
e'_i & \mbox{if $i=1$}\\
-e'_i & \mbox{if $i=-1$}
\end{array}\right..\]
Let $Q=L\,Q^\mathrm u$ be the corresponding parabolic subgroup of $K$,
so that $L=K_h\cong\mathrm{GL}(1)\times\mathrm{SO}(2n-3)\times\mathrm{GL}(1)$.

The centralizer of $e$ is $K_e=L_eQ^\mathrm u$ where
$L_e\cong\mathrm{GL}(1)\times\mathrm{SO}(2n-3)$,
the $\mathrm{GL}(1)$ factor of $L_e$ is embedded skew-diagonally, $z\mapsto(z,z^{-1})$, into the $\mathrm{GL}(1)\times\mathrm{GL}(1)$ factor of $L$.

Case (II)
\[e=e_1\otimes \varphi'_{1},\qquad
f=-e_{-1}\otimes \varphi'_{-1},\] 
\[h(e_i)=\left\{\begin{array}{cl}
e_i & \mbox{if $i=1$}\\
-e_i & \mbox{if $i=-1$}\\
0 & \mbox{otherwise}
\end{array}\right.,\quad
h(e'_i)=\left\{\begin{array}{cl}
-e'_i & \mbox{if $i=1$}\\
e'_i & \mbox{if $i=-1$}
\end{array}\right..\]
The centralizer of $h$ is the same as in case (I).

The centralizer of $e$ is also the same, except that the $\mathrm{GL}(1)$ factor of $L_e$ is embedded diagonally, $z\mapsto(z,z)$, into the $\mathrm{GL}(1)\times\mathrm{GL}(1)$ factor of $L$.

\subsubsection*{2.2. $\mathbf{(+3,+1^{2n-3},-1)}$}\

\[e=e_1\otimes (\varphi'_1-\varphi'_{-1}),\qquad
f=e_{-1}\otimes (\varphi'_{1}-\varphi'_{-1}),\] 
\[h(e_i)=\left\{\begin{array}{cl}
2e_i & \mbox{if $i=1$}\\
-2e_i & \mbox{if $i=-1$}\\
0 & \mbox{otherwise}
\end{array}\right.,\quad
h(e'_i)=0\ \forall\ i.\]
Let $Q=L\,Q^\mathrm u$ be the corresponding parabolic subgroup of $K$,
so that $L=K_h\cong\mathrm{GL}(1)\times\mathrm{SO}(2n-3)\times\mathrm{GL}(1)$.

The centralizer of $e$ is $K_e=L_eQ^\mathrm u$ where
$L_e\cong\mathrm{O}(1)\times\mathrm{SO}(2n-3)$,
the $\mathrm{O}(1)$ factor of $L_e$ is embedded diagonally into the $\mathrm{GL}(1)\times\mathrm{GL}(1)$ factor of $L$.

\subsubsection*{2.3. $\mathbf{(-3,+1^{2n-2})}$, $I$  and $II$}\

Case (I)
\[e=e_0\otimes \varphi'_{-1},\qquad
f=-2e_0\otimes \varphi'_{1},\] 
\[h(e_i)=0\ \forall\ i,\quad
h(e'_i)=\left\{\begin{array}{cl}
2e'_i & \mbox{if $i=1$}\\
-2e'_i & \mbox{if $i=-1$}
\end{array}\right..\]
Here the centralizer of $h$ is $K_h=K\cong\mathrm{SO}(2n-1)\times\mathrm{GL}(1)$.

The centralizer of $e$ is $K_e\cong\mathrm S(\mathrm{O}(2n-2)\times\mathrm O(1))$
embedded as 
\[(A,z)\mapsto((A,z),z^{-1})\] 
into $\mathrm S(\mathrm O(2n-2)\times\mathrm{O}(1))\times\mathrm{GL}(1)$,
where $\mathrm S(\mathrm O(2n-2)\times\mathrm{O}(1))$ is included in the $\mathrm{SO}(2n-1)$ factor of $K$.

Case (II)
\[e=e_0\otimes \varphi'_{1},\qquad
f=-2e_0\otimes \varphi'_{-1},\] 
\[h(e_i)=0\ \forall\ i,\quad
h(e'_i)=\left\{\begin{array}{cl}
-2e'_i & \mbox{if $i=1$}\\
2e'_i & \mbox{if $i=-1$}
\end{array}\right..\]
The centralizers of $h$ and $e$ are the same as in case (I).

\subsubsection*{2.4. $\mathbf{(+3^2,+1^{2n-5})}$}\

\[e=e_1\otimes \varphi'_{-1}-e_2\otimes \varphi'_1,\qquad
f=2(e_{-2}\otimes \varphi'_{-1}- e_{-1}\otimes \varphi'_{1}),\] 
\[h(e_i)=\left\{\begin{array}{cl}
2e_i & \mbox{if $1\leq i\leq 2$}\\
-2e_i & \mbox{if $-2\leq i\leq -1$}\\
0 & \mbox{otherwise}
\end{array}\right.,\quad
h(e'_i)=0\ \forall\ i.\]
Let $Q=L\,Q^\mathrm u$ be the corresponding parabolic subgroup of $K$,
so that $L=K_h\cong\mathrm{GL}(2)\times\mathrm{SO}(2n-5)\times\mathrm{GL}(1)$.

The centralizer of $e$ is $K_e=L_eQ^\mathrm u$ where
$L_e\cong\mathrm{SO}(2n-5)\times\mathrm{GL}(1)$,
the $\mathrm{GL}(1)$ factor of $L_e$ is embedded as 
\[z\mapsto((z,z^{-1}),z^{-1})\] 
into $(\mathrm{GL}(1)\times\mathrm{GL}(1))\times\mathrm{GL}(1)$ 
included into the $\mathrm{GL}(2)\times\mathrm{GL}(1)$ factor of $L$.

\subsection{$\mathrm{Sp}(2n)/\mathrm{GL}(n)$}\

$K=\mathrm{GL}(n)$, $n\geq2$, $\mathfrak p=V(2\omega_1)\oplus V(2\omega_{n-1})$ as $K^{\mathrm{ss}}$-module.

Let us fix a basis $e_1,\ldots,e_n$ of $\mathbb C^{n}$ and denote by $\varphi_1,\ldots,\varphi_n$ the dual basis of $(\mathbb C^{n})^*$. Then $K=\mathrm{GL}(\mathbb C^{n})$ and 
\[\mathfrak p=\mathsf S^2(\mathbb C^{n})\oplus\mathsf S^2(\mathbb C^{n})^*.\] 

\subsubsection*{3.1. $\mathbf{(+2^r,+1^{2n-2r})}$}\

\[e=\sum_{i=1}^re_i e_{r-i+1},\qquad
h(e_i)=\left\{\begin{array}{cl}
e_i & \mbox{if $1\leq i\leq r$}\\
0 & \mbox{otherwise}
\end{array}\right.,\qquad
f=\sum_{i=1}^r\varphi_{i}\varphi_{r-i+1}.\]
Let $Q=L\,Q^\mathrm u$ be the corresponding parabolic subgroup of $K$,
so that $L=K_h\cong\mathrm{GL}(r)\times\mathrm{GL}(n-r)$.

The centralizer of $e$ is $K_e=L_eQ^\mathrm u$ where
$L_e\cong\mathrm{O}(r)\times\mathrm{GL}(n-r)$.

\subsubsection*{3.2. $\mathbf{(-2^r,+1^{2n-2r})}$}\

\[e=\sum_{i=1}^r\varphi_{n-r+i}\varphi_{n-i+1},\qquad
f=\sum_{i=1}^r e_{n-r+i} e_{n-i+1},\]
\[h(e_i)=\left\{\begin{array}{cl}
-e_i & \mbox{if $n-r+1\leq i\leq n$}\\
0 & \mbox{otherwise}
\end{array}\right..\]
The centralizers of $h$ and $e$ are the same as in the previous case up to an external automorphism of $K$. 



\subsubsection*{3.3. $\mathbf{(+2^r,-2^s,+1^{2n-2r-2s})}$}\

\[e=\sum_{i=1}^r e_i e_{r-i+1} + \sum_{i=1}^s \varphi_{n-s+i}\varphi_{n-i+1},\qquad
f=\sum_{i=1}^r \varphi_i\varphi_{r-i+1} + \sum_{i=1}^s e_{n-s+i} e_{n-i+1},\]
\[h(e_i)=\left\{\begin{array}{cl}
e_i & \mbox{if $1\leq i\leq r$}\\
-e_i & \mbox{if $n-s+1\leq i\leq n$}\\
0 & \mbox{otherwise}
\end{array}\right..\]

Let $Q=L\,Q^\mathrm u$ be the corresponding parabolic subgroup of $K$,
so that $L=K_h\cong\mathrm{GL}(r)\times\mathrm{GL}(n-r-s)\times\mathrm{GL}(s)$.

The centralizer of $e$ is $K_e=L_eQ^\mathrm u$ where
$L_e\cong\mathrm{O}(r)\times\mathrm{GL}(n-r-s)\times\mathrm{O}(s)$.

\subsection{$\mathrm{SO}(2n)/\mathrm{SO}(2n-2)\times\mathrm{SO}(2)$}\

$K=\mathrm{SO}(2n-2)\times\mathrm{SO}(2)$, $n>4$, $\mathfrak p=V(\omega_1)\oplus V(\omega_1)$ as $K^{\mathrm{ss}}$-module.

Let us fix a basis $e_1,\ldots,e_{n-1},e_{-n+1},\ldots,e_{-1}$ of $\mathbb C^{2n-2}$, a symmetric bilinear form $\beta$ such that $\beta(e_i,e_j)=\delta_{i,-j}$ for all $i,j$. Similarly, let us fix a basis $e'_1,e'_{-1}$ of $\mathbb C^{2}$ and a symmetric bilinear form $\beta'$ such that $\beta'(e'_i,e'_j)=\delta_{i,-j}$ for all $i,j$. For convenience, let us denote by $\varphi'_1,\varphi'_{-1}$ the dual basis of $(\mathbb C^{2})^*$. Then $K=\mathrm{SO}(\mathbb C^{2n-2},\beta)\times\mathrm{SO}(\mathbb C^{2},\beta')$ and 
\[\mathfrak p=\mathbb C^{2n-2}\otimes(\mathbb C^{2})^*.\] 

\subsubsection*{4.1. $\mathbf{(+2^2,+1^{2n-4})}$, $I$  and $II$}\

Case (I)
\[e=e_1\otimes \varphi'_{-1},\qquad
f=-e_{-1}\otimes \varphi'_{1},\] 
\[h(e_i)=\left\{\begin{array}{cl}
e_i & \mbox{if $i=1$}\\
-e_i & \mbox{if $i=-1$}\\
0 & \mbox{otherwise}
\end{array}\right.,\quad
h(e'_i)=\left\{\begin{array}{cl}
e'_i & \mbox{if $i=1$}\\
-e'_i & \mbox{if $i=-1$}
\end{array}\right..\]
Let $Q=L\,Q^\mathrm u$ be the corresponding parabolic subgroup of $K$,
so that $L=K_h\cong\mathrm{GL}(1)\times\mathrm{SO}(2n-4)\times\mathrm{GL}(1)$.

The centralizer of $e$ is $K_e=L_eQ^\mathrm u$ where
$L_e\cong\mathrm{GL}(1)\times\mathrm{SO}(2n-4)$,
the $\mathrm{GL}(1)$ factor of $L_e$ is embedded skew-diagonally, $z\mapsto(z,z^{-1})$, into the $\mathrm{GL}(1)\times\mathrm{GL}(1)$ factor of $L$.

Case (II)
\[e=e_1\otimes \varphi'_{1},\qquad
f=-e_{-1}\otimes \varphi'_{-1},\] 
\[h(e_i)=\left\{\begin{array}{cl}
e_i & \mbox{if $i=1$}\\
-e_i & \mbox{if $i=-1$}\\
0 & \mbox{otherwise}
\end{array}\right.,\quad
h(e'_i)=\left\{\begin{array}{cl}
-e'_i & \mbox{if $i=1$}\\
e'_i & \mbox{if $i=-1$}
\end{array}\right..\]
The centralizer of $h$ is the same as in case (I).

The centralizer of $e$ is also the same, except that the $\mathrm{GL}(1)$ factor of $L_e$ is embedded diagonally, $z\mapsto(z,z)$, into the $\mathrm{GL}(1)\times\mathrm{GL}(1)$ factor of $L$.

\subsubsection*{4.2. $\mathbf{(+3,+1^{2n-4},-1)}$}\

\[e=e_1\otimes (\varphi'_1-\varphi'_{-1}),\qquad
f=e_{-1}\otimes (\varphi'_{1}-\varphi'_{-1}),\] 
\[h(e_i)=\left\{\begin{array}{cl}
2e_i & \mbox{if $i=1$}\\
-2e_i & \mbox{if $i=-1$}\\
0 & \mbox{otherwise}
\end{array}\right.,\quad
h(e'_i)=0\ \forall\ i.\]
Let $Q=L\,Q^\mathrm u$ be the corresponding parabolic subgroup of $K$,
so that $L=K_h\cong\mathrm{GL}(1)\times\mathrm{SO}(2n-4)\times\mathrm{GL}(1)$.

The centralizer of $e$ is $K_e=L_eQ^\mathrm u$ where
$L_e\cong\mathrm{O}(1)\times\mathrm{SO}(2n-4)$,
the $\mathrm{O}(1)$ factor of $L_e$ is embedded diagonally into the $\mathrm{GL}(1)\times\mathrm{GL}(1)$ factor of $L$.

\subsubsection*{4.3. $\mathbf{(-3,+1^{2n-3})}$, $I$  and $II$}\

Case (I)
\[e=(e_{n-1}-e_{-n+1})\otimes \varphi'_{-1},\qquad
f=(e_{n-1}-e_{-n+1})\otimes \varphi'_{1},\] 
\[h(e_i)=0\ \forall\ i,\quad
h(e'_i)=\left\{\begin{array}{cl}
2e'_i & \mbox{if $i=1$}\\
-2e'_i & \mbox{if $i=-1$}
\end{array}\right..\]
Here the centralizer of $h$ is $K_h=K\cong\mathrm{SO}(2n-2)\times\mathrm{GL}(1)$.

The centralizer of $e$ is $K_e\cong\mathrm{SO}(2n-3)\times\mathrm O(1)$,
embedded as 
\[(A,z)\mapsto((A,z),z^{-1})\] 
into $(\mathrm{SO}(2n-3)\times\mathrm{O}(1))\times\mathrm{GL}(1)$,
where $\mathrm{SO}(2n-3)\times\mathrm{O}(1)$ is included in the $\mathrm{SO}(2n-2)$ factor of $K$.

Case (II)
\[e=(e_{n-1}-e_{-n+1})\otimes \varphi'_{1},\qquad
f=(e_{n-1}-e_{-n+1})\otimes \varphi'_{-1},\] 
\[h(e_i)=0\ \forall\ i,\quad
h(e'_i)=\left\{\begin{array}{cl}
-2e'_i & \mbox{if $i=1$}\\
2e'_i & \mbox{if $i=-1$}
\end{array}\right..\]
The centralizers of $h$ and $e$ are the same as in case (I).

\subsubsection*{4.4. $\mathbf{(+3^2,+1^{2n-6})}$}\

\[e=e_1\otimes \varphi'_{-1}-e_2\otimes \varphi'_1,\qquad
f=2(e_{-2}\otimes \varphi'_{-1}- e_{-1}\otimes \varphi'_{1}),\] 
\[h(e_i)=\left\{\begin{array}{cl}
2e_i & \mbox{if $1\leq i\leq 2$}\\
-2e_i & \mbox{if $-2\leq i\leq -1$}\\
0 & \mbox{otherwise}
\end{array}\right.,\quad
h(e'_i)=0\ \forall\ i.\]
Let $Q=L\,Q^\mathrm u$ be the corresponding parabolic subgroup of $K$,
so that $L=K_h\cong\mathrm{GL}(2)\times\mathrm{SO}(2n-6)\times\mathrm{GL}(1)$.

The centralizer of $e$ is $K_e=L_eQ^\mathrm u$ where
$L_e\cong\mathrm{SO}(2n-6)\times\mathrm{GL}(1)$,
the $\mathrm{GL}(1)$ factor of $L_e$ is embedded as 
\[z\mapsto((z,z^{-1}),z^{-1})\] 
into $(\mathrm{GL}(1)\times\mathrm{GL}(1))\times\mathrm{GL}(1)$ 
included into the $\mathrm{GL}(2)\times\mathrm{GL}(1)$ factor of $L$.

\subsection{$\mathrm{SO}(2n)/\mathrm{GL}(n)$}\

$K=\mathrm{GL}(n)$, $n\geq4$, $\mathfrak p=V(\omega_2)\oplus V(\omega_{n-2})$ as $K^{\mathrm{ss}}$-module.

Let us fix a basis $e_1,\ldots,e_n$ of $\mathbb C^{n}$ and denote by $\varphi_1,\ldots,\varphi_n$ the dual basis of $(\mathbb C^{n})^*$. Then $K=\mathrm{GL}(\mathbb C^{n})$ and 
\[\mathfrak p=\mathsf \Lambda^2(\mathbb C^{n})\oplus\mathsf \Lambda^2(\mathbb C^{n})^*.\] 

\subsubsection*{5.1. $\mathbf{(+2^r,+1^{n-2r})}$}\

\[e=\sum_{i=1}^re_i\wedge e_{2r-i+1},\qquad
h(e_i)=\left\{\begin{array}{cl}
e_i & \mbox{if $1\leq i\leq 2r$}\\
0 & \mbox{otherwise}
\end{array}\right.,\qquad
f=\sum_{i=1}^r\varphi_{i}\wedge\varphi_{2r-i+1}.\]
Let $Q=L\,Q^\mathrm u$ be the corresponding parabolic subgroup of $K$,
so that $L=K_h\cong\mathrm{GL}(2r)\times\mathrm{GL}(n-2r)$.

The centralizer of $e$ is $K_e=L_eQ^\mathrm u$ where
$L_e\cong\mathrm{Sp}(2r)\times\mathrm{GL}(n-2r)$.

\subsubsection*{5.2. $\mathbf{(-2^r,+1^{n-2r})}$}\

\[e=\sum_{i=1}^r\varphi_{n-2r+i-1}\wedge \varphi_{n-i+1},\qquad
f=\sum_{i=1}^r e_{n-2r+i-1}\wedge e_{n-i+1},\]
\[h(e_i)=\left\{\begin{array}{cl}
-e_i & \mbox{if $n-2r+1\leq i\leq n$}\\
0 & \mbox{otherwise}
\end{array}\right..\]
The centralizers of $h$ and $e$ are the same as in the previous case up to an external automorphism of $K$. 



\subsubsection*{5.3. $\mathbf{(+2^r,-2^s,+1^{n-2r-2s})}$}\

\[e=\sum_{i=1}^re_i\wedge e_{2r-i+1} + \sum_{i=1}^s\varphi_{n-2s+i-1}\wedge \varphi_{n-i+1},\]
\[h(e_i)=\left\{\begin{array}{cl}
e_i & \mbox{if $1\leq i\leq 2r$}\\
-e_i & \mbox{if $n-2s+1\leq i\leq n$}\\
0 & \mbox{otherwise}
\end{array}\right.,\]
\[f=\sum_{i=1}^r\varphi_{i}\wedge\varphi_{2r-i+1} + \sum_{i=1}^s e_{n-2s+i-1}\wedge e_{n-i+1}.\]
Let $Q=L\,Q^\mathrm u$ be the corresponding parabolic subgroup of $K$,
so that $L=K_h\cong\mathrm{GL}(2r)\times\mathrm{GL}(n-2r-2s)\times\mathrm{GL}(2s)$.

The centralizer of $e$ is $K_e=L_eQ^\mathrm u$ where
$L_e\cong\mathrm{Sp}(2r)\times\mathrm{GL}(n-2r-2s)\times\mathrm{Sp}(2s)$.

\subsubsection*{5.4. $\mathbf{(+3,+1^{n-3})}$}\

\[e=e_1\wedge e_2 + \varphi_2\wedge \varphi_n,\qquad
f=2(\varphi_1\wedge \varphi_2 + e_2\otimes e_n),\] 
\[h(e_i)=\left\{\begin{array}{cl}
2e_i & \mbox{if $i=1$}\\
-2e_i & \mbox{if $i=n$}\\
0 & \mbox{otherwise}
\end{array}\right..\]
Let $Q=L\,Q^\mathrm u$ be the corresponding parabolic subgroup of $K$,
so that $L=K_h\cong\mathrm{GL}(1)\times\mathrm{GL}(n-2)\times\mathrm{GL}(1)$.

The centralizer of $e$ is $K_e=L_eQ^\mathrm u$ where
$L_e\cong\mathrm{GL}(1)\times\mathrm{GL}(n-3)$, $L_e$ is embedded as 
\[(z,A)\mapsto(z,(z^{-1},A),z)\] 
into $\mathrm{GL}(1)\times(\mathrm{GL}(1)\times\mathrm{GL}(n-3))\times\mathrm{GL}(1)$,
and $(\mathrm{GL}(1)\times\mathrm{GL}(n-3))$ is included into the $\mathrm{GL}(n-2)$ factor of $L$.

\section{Tables of spherical nilpotent $K$-orbits  in $\mathfrak p$ in the classical Hermitian cases}\label{B}

In Tables~1--6, for every spherical nilpotent orbit $Ke \subset \mathfrak p$, we report its signed partition (column~2), the Kostant-Dynkin diagram and the height of $Ge$ (columns~3 and 4), the Kostant-Dynkin diagram and the $\mathfrak p$-height of $Ke$ (columns~5 and 6), the codimension of $\ol{Ke} \smallsetminus Ke$ in $\ol{Ke}$ (column~7) and the weight semigroup of $\ol{Ke}$ (column~8).

The generators of the weight semigroups given in the tables are expressed in terms of the fundamental weights of $K^{\mathrm{ss}}$, the semisimple part of $K$, plus a multiple of $\chi$, where $\chi$ denotes the character of the 1-dimensional center of $K$ on $\mathfrak p_1$, as defined in Section~\ref{s:4}.

Recall that the fundamental weights of an irreducible root system of rank $n$ are denoted by $\omega_1,\ldots,\omega_n$ (and enumerated as in Bourbaki). For notational convenience, we set $\omega_i=0$ if $i\leq 0$ or $i>n$.

In Tables~7--11, for every spherical nilpotent orbit $Ke$ in $\gop$, we report the Luna diagram and the set of spherical roots of the spherical system of $K\pi(e)$, where the definition of $\pi(e)$ is as follows. Recall that in the Hermitian case $\mathfrak p= \mathfrak p_1 \oplus \mathfrak p_2$,  so that $e=e_1+e_2$ with $e_1\in\mathfrak p_1$ and $e_2\in\mathfrak p_2$. Therefore, we set $\pi(e)=[e_1]\in\mathbb P(\mathfrak p_1)$ if $e_2=0$,   $\pi(e)=[e_2]\in\mathbb P(\mathfrak p_2)$ if $e_1=0$, and $\pi(e)=([e_1],[e_2])\in\mathbb P(\mathfrak p_1)\times\mathbb P(\mathfrak p_2)$ otherwise.

\includepdf[fitpaper,pages=-,landscape]{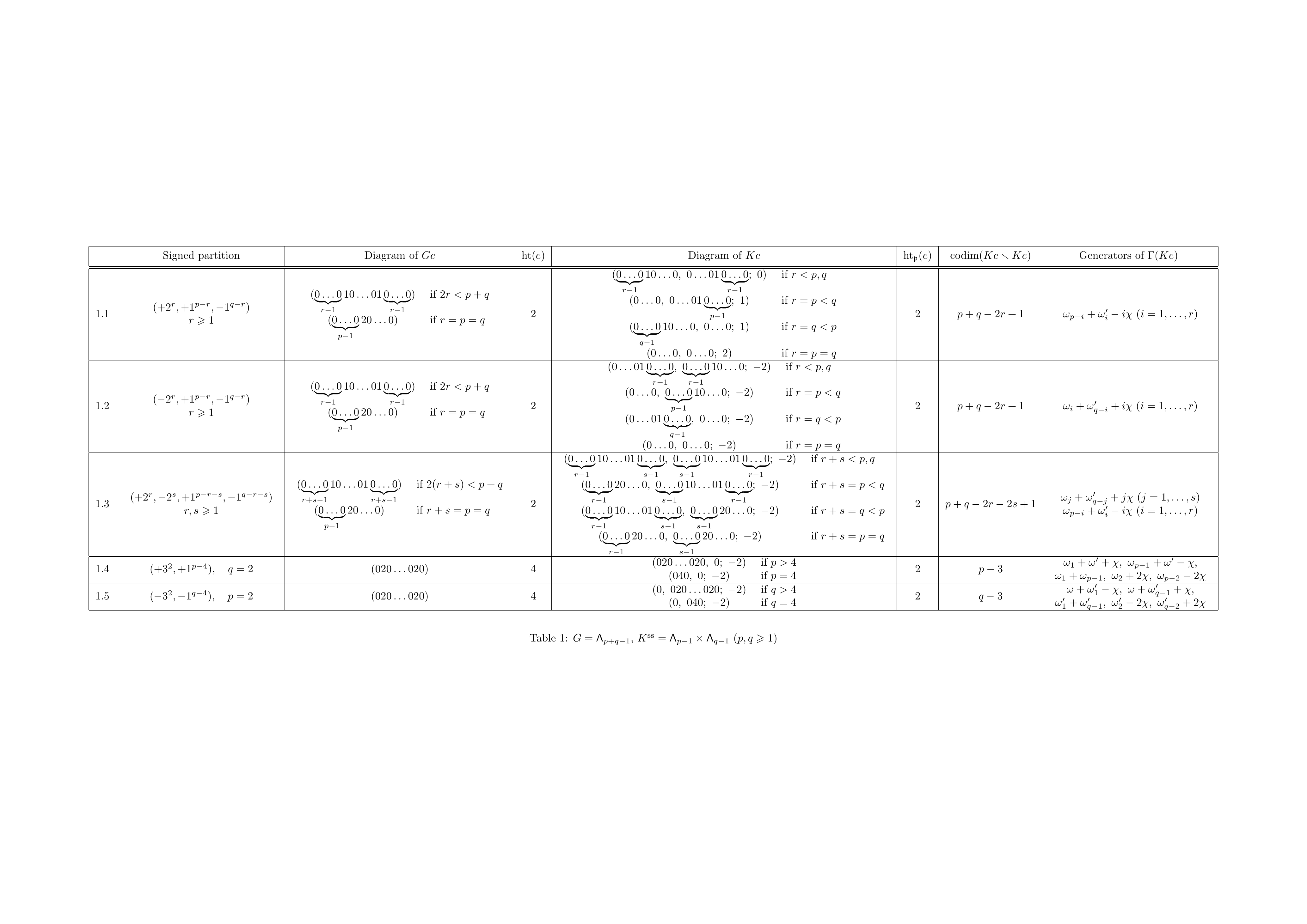}
\includepdf[fitpaper,pages=-,landscape]{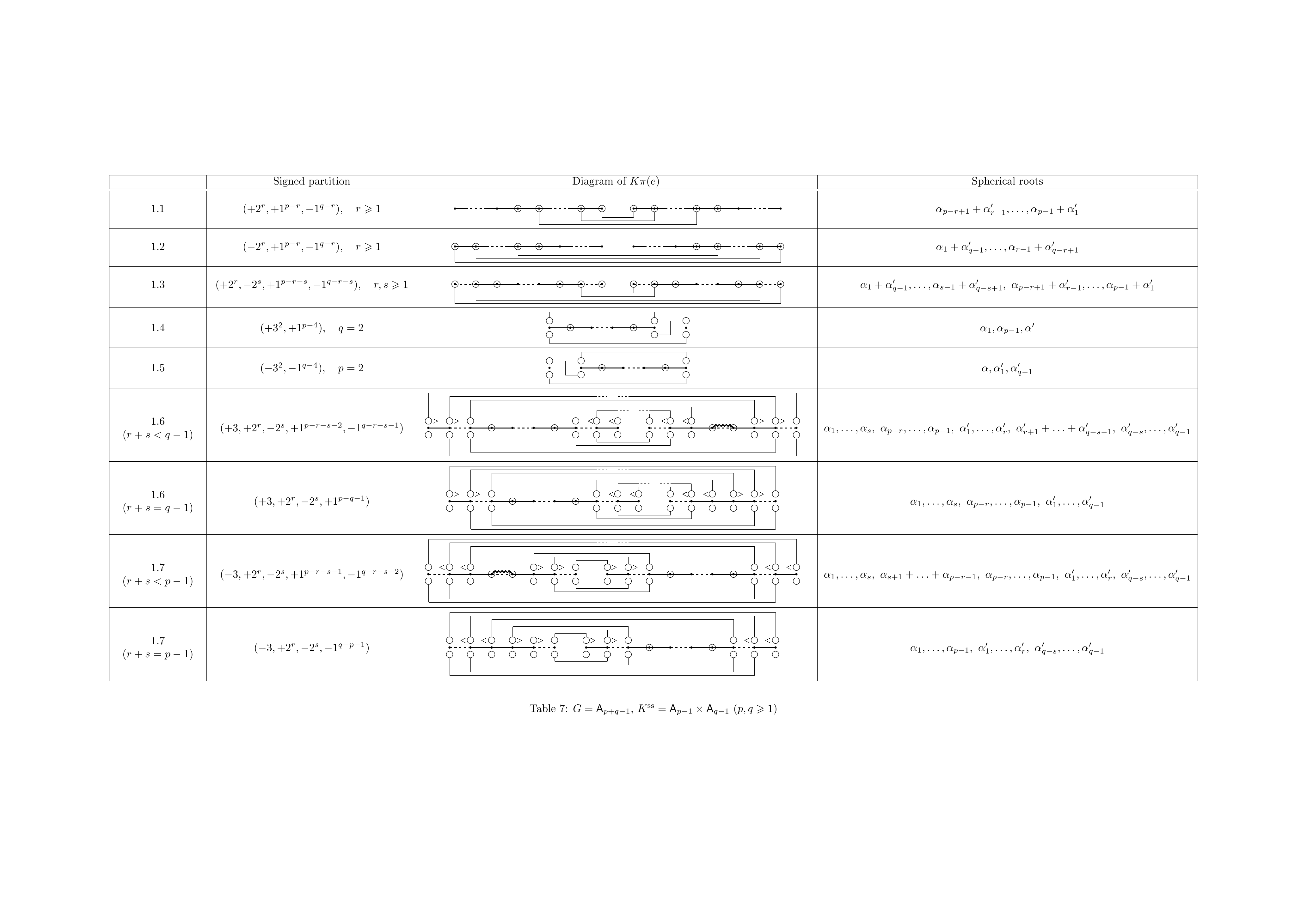}


\end{document}